\documentclass[12pt, dvipdfmx]{amsart}
\usepackage{graphicx} 
\usepackage{amsmath}
\usepackage{amssymb}
\usepackage{amscd}
\usepackage{amsthm}
\usepackage{enumerate}
\usepackage{xcolor}
\usepackage[all]{xy}
\usepackage{url}
\usepackage[colorlinks=true,linkcolor=blue,breaklinks,pagebackref]{hyperref}
\usepackage[colorlinks=true]{hyperref}
\usepackage[lite,abbrev,alphabetic]{amsrefs}
\usepackage[poorman]{cleveref}
\usepackage{comment}
\usepackage{todonotes}
\usepackage{tikz-cd}
\usepackage{quiver}

\theoremstyle{definition}
\newtheorem{definition}{Definition}[section]
\newtheorem{example}[definition]{Example}

\theoremstyle{plain}
\newtheorem{theorem}[definition]{Theorem}
\newtheorem{proposition}[definition]{Proposition}
\newtheorem{lemma}[definition]{Lemma}
\newtheorem{corollary}[definition]{Corollary}

\newtheorem*{theorem*}{Theorem}
\newtheorem{construction}[definition]{Construction}

\theoremstyle{remark}
\newtheorem{remark}[definition]{Remark}

\crefname{theorem}{Theorem}{Theorems}
\crefname{proposition}{Proposition}{Propositions}
\crefname{lemma}{Lemma}{Lemmas}
\crefname{corollary}{Corollary}{Corollaries}
\crefname{conjecture}{Conjecture}{Conjectures}
\crefname{hypothesis}{Hypothesis}{Hypotheses}
\crefname{remark}{Remark}{Remarks}
\crefname{condition}{Condition}{Conditions}
\crefname{example}{Example}{Examples}

\def\C{\mathbb{C}}
\def\R{\mathbb{R}}
\def\Z{\mathbb{Z}}
\def\A{\mathbb{A}}

\def\PGL{\mathrm{PGL}}

\def\rd{\,\mathrm{d}}

\def\Cc{C_c^{\infty}}

\DeclareMathOperator{\vol}{vol}

\DeclareMathOperator{\Lie}{Lie}

\DeclareMathOperator{\Irr}{Irr}
\DeclareMathOperator{\re}{Re}

\DeclareMathOperator{\tr}{tr}

\newcommand{\fa}{\mathfrak{a}}

\newcommand{\fg}{\mathfrak{g}}
\newcommand{\fh}{\mathfrak{h}}

\newcommand{\ft}{\mathfrak{t}}

\newcommand{\bG}{\mathbb{G}}

\newcommand{\cA}{\mathcal{A}}
\newcommand{\cB}{\mathcal{B}}

\newcommand{\cF}{\mathcal{F}}

\newcommand{\cL}{\mathcal{L}}

\newcommand{\cO}{\mathcal{O}}
\newcommand{\cP}{\mathcal{P}}

\newcommand{\cT}{\mathcal{T}}
\newcommand{\cU}{\mathcal{U}}

\newcommand{\cX}{\mathcal{X}}

\numberwithin{equation}{section}
\newcommand{\disc}{\mathrm{disc}}

\newcommand{\cusp}{\mathrm{cusp}}

\newcommand{\temp}{\mathrm{temp}}

\newcommand{\unit}{\mathrm{unit}}

\newcommand{\herm}{\mathrm{herm}}

\newcommand{\abs}[1]{\lvert #1 \rvert}

\newcommand{\st}{\, \vert \,}

\DefineSimpleKey{bib}{arxiv}
\newcommand\myurl[1]{\url{#1}}
\newcommand\arxiv[1]{available at \href{https://arxiv.org/abs/#1}{arXiv:#1}}

\BibSpec{article}{%
  +{}  {\PrintAuthors}                {author}
  +{,\,\,} {\textit}                     {title}
  +{.} { }                            {part}
  +{:} {\textit}                     {subtitle}
  +{,} {\PrintContributions}         {contribution}
  +{,} {\PrintConference}            {conference}
  +{}  {\PrintBook}                   {book}
  +{,} { }                            {booktitle}
  +{,} { }                            {journal}
  +{,\,\,} {\textbf}                     {volume}
  +{,\,\,} {\issuetext}                     {number}
  +{,\,\,} {\PrintDateB }                 {date}
  +{,\,\,} {pp.}                        {pages}
  +{,\,\,} {available at \eprint}        {eprint}
  +{;} { \arxiv}               {arxiv}
  +{.\!\!}  {\SentenceSpace \PrintReviews} {review}
  +{,\,\,} {\PrintDOI}                   {doi}
}

\title{A note on the Sauvageot density principle}
\author{Yugo Takanashi} 
\address{Graduate School of Mathematical Sciences, University of Tokyo, 3-8-1 Komaba, Meguro-Ku, Tokyo 153-8914, Japan}
\email{tknashi@ms.u-tokyo.ac.jp}
\date{}

\begin{document}

\begin{abstract}
    In this short note, we address a gap in the proof of Sauvageot's paper \cite{Sau97} pointed out in \cite{NelsonVenkatesh2021-orbitmethod} and provide a complete proof of its main theorem.
\end{abstract}

\maketitle

\section{Introduction}

The Sauvageot density principle for reductive groups over local fields was proved in \cite{Sau97} and it has been applied to prove well known results; see, for example, \cite{Shi12}, \cite{FLM15} and \cite{MikrosBergeron2017}.
It was noted in \cite{Shi12}*{Appendix A} that there are some minor errors in the proof of this result.
More recently, in \cite{NelsonVenkatesh2021-orbitmethod}*{p.159, footnote}, it was pointed out that there is a gap in the proof. 
In this short note, we fill the gap in the proof pointed out in \cite{NelsonVenkatesh2021-orbitmethod}.
We also fix some notational and logical errors in the proof, pointed out in \cite{Shi12}.

Roughly speaking, to prove the density principle, we need to approximate a function on the unitary dual using the Fourier transform of compactly supported functions. 
The proof given in \cite{Sau97} utilizes only the functions that factor through the space of Bernstein components. 
However, the map from the unitary dual to the union of Bernstein components has non-trivial fibers, and the functions used by Sauvageot cannot separate the points within a single fiber.
We will explain why this method by Sauvageot actually fails, see \cref{remark:counterexample}.
This appears to be the gap identified in \cite{NelsonVenkatesh2021-orbitmethod}*{p.159, footnote}.

The main idea to fill the gap is to utilize a module version of the Stone-Weierstrass theorem, see \cref{lemma:module-version-Stone-Weierstrass}, and the measure theoretic analog of this result, see \cref{theorem:main-theorem}.
This enables us to distinguish points within each fiber of the map, allowing us to recover Sauvageot's original result.

In the last part of this paper, we will present two applications of Sauvageot's density principle.
One is the automorphic Plancherel density theorem by Shin.
We will give a somewhat extended version of his result.
The other one is an application of the Weyl law by Finis-Matz and Eikemeier, which produces spherical cusp forms of a quasi-split group with prescribed local properties.

\subsection{Notation and assumptions}\label{subsection:notation}

Let $F$ be a local field or a global field of characteristic 0.
Let $\bG$ be a connected reductive group over $F$.
We denote by $G$ the set of $F$-valued points of $\bG$.
Unless otherwise stated, we will identify $\bG$ and $G = \bG(F)$.
Also, we will only consider the sets of $F$-valued points of $F$-subgroups of $\bG$ and we refer to them simply as subgroups.

We use the following standard notation.
\begin{itemize}
  \item    
  Let $A_G$ denote the maximal split torus of the center $Z_G$ of $G$.
  \item Let $\cP^G(M)$ denote the set of parabolic subgroups of $G$ with the Levi factor $M$.
  \item
  Let $\cX(G)$ denote the group of unramified characters of $G$.
  Let $\cX_{\unit}(G)$ denote the group of unitary unramified characters of $G$.
  \item 
  Let $X^{*}(M)$ be the group of rational characters of $M$ over $F$.
  We set $\fa^*_M = X^*(M) \otimes_{\Z} \R$ and $\fa^*_{M, \C} = X^*(M) \otimes_{\Z} \C$.
  We have the natural map 
  \begin{align*}
        \fa^*_{M, \C} \to \cX(M) \colon \lambda \mapsto \chi_{\lambda}.
  \end{align*}
  \item Let $W(G, M)$ denote the relative Weyl group for $M$. 
\end{itemize}

If $F = \R$, let $\Lie(G)$ or $\fg$ denote the Lie algebra of $G$.
We always fix a maximal compact subgroup of $G$ and we denote it by $K$.
By a representation of a real reductive group, we mean an admissible $(\fg \otimes_{\R} \C, K)$-module over $\C$. 
If we consider a topological representation, we take the Casselman-Wallach globalization.
If $F$ is a $p$-adic field, by a representation of $p$-adic groups, we mean a smooth admissible representation over $\C$.
In both cases, for a parabolic subgroup $P$ of $G$ and an irreducible representation $\sigma$ of $M$, let $i^G_P(\sigma)$ denote the normalized parabolic induction of $\sigma$.

\subsection{Acknowledgment}
I became aware of Sauvageot's paper while collaborating with Satoshi Wakatsuki.
I would like to thank Satoshi Wakatsuki for his encouragement and for careful reading.
I would like to thank Sug Woo Shin for his encouragement, careful reading and helpful comments which greatly improved the readability of this article.
Some parts of this paper are based on his work \cite{Shi12}.
I would like to thank Masaki Natori and Mao Hoshino for their careful reading and for pointing out some mistakes in \cref{definition:topology} and \cref{lemma:closure-invariant}.
I would like to thank Yoichi Mieda for his careful reading.
Also, I would like to thank the anonymous referees for their careful reading and helpful comments.
This work is supported by JSPS Grant-in-Aid for JSPS Fellows 23KJ0403.

\section{The space of infinitesimal characters}

\subsection{Definition}

We first recall the elementary facts about the space of infinitesimal characters, following \cite{BernsteinDeligneKazhdan1986}*{2}. 

Let $F, G$ be as in \cref{subsection:notation}.

The notion of infinitesimal characters for reductive groups over $\R$ is well known;  they correspond to the characters of the commutative algebra $Z(U(\fg_{\C}))$.
Let $\Theta(G)$ denote the set of infinitesimal characters of $G$ .

\begin{definition}
   We assume that $F$ is a $p$-adic field.
   We call a pair $(M, \sigma)$, consisting of a Levi subgroup $M$ of $G$ and a supercuspidal representation $\sigma$ of $M$ a cuspidal datum of $G$. 
   We refer to a $G$-conjugacy class of a cuspidal datum of $G$ as an infinitesimal character of $G$.
   We write $\Theta(G)$ for the set of infinitesimal characters.
   We write $[M, \sigma]_G$ for the class of $(M, \sigma)$.
\end{definition}

The following subrepresentation theorem is well known.

\begin{theorem}[\cite{Renard2010-p-adiques}*{VI 5.4, Th\'eor\`eme}]\label{theorem:subrepresentation-theorem}
   We assume that $F$ is a $p$-adic field.
   Let $\pi$ be an irreducible representation of $G$.
   \begin{enumerate}
    \item There exists a cuspidal datum $(M, \sigma)$ and a parabolic subgroup $P \in \cP^G(M)$ such that $\pi \subset i^G_{P}(\sigma)$. 
    \item Let $(M, \sigma)$ and $(M', \sigma')$ be two cuspidal data.
    We assume that there exist parabolic subgroups $P \in \cP^G(M)$ and $P' \in \cP^G(M)$ such that the parabolic inductions $i^G_P(\sigma)$ and $i^G_{P'}(\sigma')$ have a common irreducible component.
    Then, the infinitesimal characters $[M, \sigma]_G$ and $[M', \sigma']_G$ are equal.
    Also, the converse holds true.
    \item      
    In the case of $(2)$, the semisimplifications of $i^G_P(\sigma)$ and $i^G_{P'}(\sigma')$ are equal.
   \end{enumerate} 
\end{theorem}

\begin{construction}\label{construction:space-infinitesimal-character}
    We assume that $F$ is a $p$-adic field.
    Let $(M, \sigma)$ be a cuspidal datum of $G$.
    We call the image of the map 
    \begin{align}
       \cX(M) \to \Theta(G) \colon \chi \mapsto [M, \sigma \otimes \chi]_G
    \end{align}
    a connected component of $\Theta(G)$ containing $[M, \sigma]_G$.
    We denote it by $C = C_{[M, \sigma]}$.
    We denote the stabilizer of $\sigma$ in $\cX(M)$ by $\cX(M)_{\sigma}$.
    Also, we denote the stabilizer of $C = C_{[M, \sigma]}$ in $W(G, M)$ by $W(G, M)_C$.

    Through the above map, we can identify the set $C_{[M, \sigma]}$ with the quotient of the complex algebraic torus $\cX(M)/\cX(M)_{\sigma}$ by the finite group $W(G, M)_{C}$.
    Thus, we endow the set $C = C_{[M, \sigma]}$ with the structure of affine algebraic variety $(\cX(M)/\cX(M)_{\sigma})//W(G, M)_{C}$.
    We also endow the set $\Theta(G)$ with the structure of a countable disjoint  union of complex algebraic varieties as the countable disjoint union of connected components.
    Note that we can identify the connected component $C$ with the image of the map 
    \begin{align}
          \fa^{*}_{M, \C} \to \Theta(G) \colon \lambda \mapsto [M, \sigma \otimes \chi_{\lambda}]_G.
    \end{align}
    This definition can be used to define the similar notion in the archimedean case.
\end{construction}

\subsection{A ring of functions and hermitian involutions}

\begin{definition}
     We assume that $F$ is a $p$-adic field.
     Let $\cA(G)$ denote the ring of regular functions on $\Theta(G)$ supported on finite connected components of $\Theta(G)$.
     We assume that $F = \R$.
     Let $\theta$ be a Cartan involution of $G(\R)$.
     In this case, let $\fh \subset \Lie(G)$ be the Lie algebra of a $\theta$-stable maximally split maximal torus $H \subset G$.
     Then, the subalgebra decomposes as $\fa_0 \oplus \ft_0$ with respect to the Cartan involution.
     The space $\fa_0$ (resp. $\ft_0$) is the $-1$-eigenspace (resp. $+1$-eigenspace) with respect to the action of $\theta$.
     We set $\fh_{\R} = \fa_0 \oplus i\ft_0$.
     Then, we have the Fourier-Laplace transform 
     \begin{align*}
         \Cc(\fh_{\R})^{W(G(\C),H(\C))} 
         \to   
         \mathcal{PW}(\fh_{\R} \otimes_{\R} \C)^{W(G(\C),H(\C))},
     \end{align*}
     from the space of compactly supported functions to the space of  Paley-Wiener functions.
     We denote by $\cA(G)$  the ring of functions on the right hand side.
\end{definition}

\begin{definition}
     We assume that $F$ is a $p$-adic field.
     We define the anti-holomorphic involution on the set $\Theta(G)$ by $[M, \sigma]_G \mapsto [M, \sigma^{h}]_G$, where $\sigma^h = \overline{\sigma}^{\vee}$. 
     We denote the image by $[M, \sigma]_G^h$.
     Let $\Theta(G)_{\herm}$ denote the set of fixed points of this involution.
     We assume that $F = \R$.
     The linear bijection $\Lie(G) \to \Lie(G) \colon X \mapsto -X$ extends $\C$-semilinearly to the semilinear isomorphism $Z(U(\Lie(G)_{\C})) \to Z(U(\Lie(G)_{\C}))$.
     We denote this map by $z \mapsto z^h$.
     This involution induces the anti-holomorphic  involution $(\cdot)^h \colon \Theta(G) \to \Theta(G)$.
\end{definition}

\begin{lemma}\label{lemma:stable-under-complex-conjugate-algebra}
     We denote the pullback of an element $a \in \cA(G)$ by the map $h$ by $a^h$.
     Then, the map $a \mapsto \overline{a^h}$ induces the $\C$-linear involution of the algebra $\cA(G)$.
\end{lemma}

\section{The space of induced representations}

We introduce the space of induced representations from essentially discrete series representations, which is similar to the space of infinitesimal characters.

\begin{definition}
    We call a pair $(M, \sigma)$ which consists of a Levi subgroup $M$ of $G$ and an essentially discrete series representation $\sigma$ of $M$ a discrete pair of $G$.  
    We refer to a $G$-conjugacy class of a discrete pair of $G$ as a discrete datum of $G$.
    We denote the class of $(M, \sigma)$ by $\{M, \sigma\}_G$ and the set of discrete data of $G$ by $\Theta_{\disc}(G)$.
\end{definition}

\begin{definition}
    Let $P = MN$ be a parabolic subgroup of $G$.
    Let $\sigma$ be a tempered irreducible representation of $M$. 
    Let $\nu \in \fa^{*}_M$ be a regular and positive with respect to $P$.
    Then, we can construct a parabolic induction $i^{G}_{P}(\sigma \otimes \chi_{\nu}) = i^{G}_{P}(\sigma_{\nu})$.
    We call such a module, a standard module of $G$.
    It has a unique irreducible quotient $J_{P}(\sigma, \nu)$.
    
    It is well-known that we have 
    \begin{align*}
        J_{P_1}(\sigma_1, \nu_1) \cong J_{P_2}(\sigma_2, \nu_2)
    \end{align*}
    if and only if the data $(P_i, \sigma_i, \nu_i)$ are conjugated by an element of $G$.
    We identify a standard module $i^{G}_{P}(\sigma_{\nu})$ and the $G$-conjugacy class of the datum $(P, \sigma, \nu)$.
    Let $\mathrm{StdMod}_G$ be the set of the data.
    Note that we can choose a canonical element in a class by assuming $P$ standard.
 \end{definition}

\begin{lemma}\label{lemma:standard-module-equivalence}
    Let $(M, \sigma)$ and $(M', \sigma')$ be discrete data of $G$ and $P \in \cP^G(M)$ and $P' \in \cP^G(M')$.
    Then, the semisimplifications $i_P^G(\sigma)$ and $i_{P'}^G(\sigma')$ are equal if and only if the discrete data $\{M, \sigma\}_G$ and $\{ M', \sigma'\}_G$ are equal.
\end{lemma}

\begin{proof}
    First, by the character formula for parabolic induction \cite{Lipsman1971}*{Theorem 9.1} and  \cite{vanDijk1972}*{5.3}, the character of $i_P^G(\sigma)$ is independent of the choice of $P$.
    Hence, we assume the real part of the central character $\re(\chi_{\sigma})$ of $\sigma$ is positive with respect to $P$ and $P$ is standard.
    Here, the positivity means that $\re(\chi_{\sigma})$ is in the closure of the cone defined by the fundamental weight associated to $P$ (that is the relative Weyl chamber associated to $P$). 

    Now, we take the Levi subgroup $M_1$ which contains $M$ such that $\re(\chi_{\sigma}) \in \fa_{M_1}^*$ and regular in $\fa_{M_1}^{*}$.
    We may assume that $M_1$ and $M$ are standard and also $\mathrm{Re}(\chi_{\sigma})$ is positive.
    Then, we obtain the standard parabolic subgroup $P_1$ of $G$ with Levi subgroup $M_1$ and $\re(\chi_{\sigma})$ is positive with respect to $P_1$.
    Also, we write the representation  $\sigma$ as $\sigma_{u} \otimes \re(\chi_{\sigma})$, where $\sigma_{u}$ is a unitary discrete series representation.
    In this situation, we have $i_P^{G}(\sigma) = i_{P_1}^G(i^{M_1}_{P \cap M_1}(\sigma_u) \otimes \re(\chi_{\sigma}))$.
    By decomposing $i^{M_1}_{P \cap M_1}(\sigma_u)$ to the irreducible tempered representations and taking their semisimplifications, we obtain the sum of standard modules.
    By the remark in the proof of \cite{Clo86}*{Proposition 2.1} and \cite{Hecht1979}*{Proposition 6.2}, the standard modules form a finite basis of the Grothendieck group of representations $\pi'$ of $G$ with $\inf_G(\pi') = \inf_G(i^G_P(\sigma))$.

    If we apply this procedure to the representation $i_{P'}^G(\sigma')$ to obtain $M_1'$, $P_1'$ and so on.
    By assumption, we see that $P_1=P_1'$, $M_1=M_1'$, $\re(\chi_{\sigma})=\re(\chi_{\sigma'})$ and $(M, \sigma_u)$ and $(M', \sigma'_u)$ are conjugate by an element $w \in M_1(F)$, i.e. $\sigma'_u = w (\sigma_u)$ and $M'=w(M)$. 
    As the character $\re(\chi_{\sigma})$ is in $\fa_{M_1}^*$, the element $w$ acts this element trivially and we have $w(M, \sigma)=(M', \sigma')$.
    Hence the result.
\end{proof}

\begin{corollary}[Corollary of the proof]
     Let $(M, \sigma)$ and $(M', \sigma')$ be discrete data of $G$ and $P \in \cP^G(M)$ and $P' \in \cP^G(M')$.
     We consider the semisimplifications of the parabolic inductions $i_P^G(\sigma)$ and $i_{P'}^G(\sigma')$ and write them as integral linear combinations of standard modules.
     If they have a common standard module, then the data $(M, \sigma)$ and $(M', \sigma')$ are equivalent.
     We first assume that $F$ is a $p$-adic field.
\end{corollary}

\begin{remark}
    We identify $\{M, \sigma\}_G$ with the semisimplification $i^G_{P}(\sigma)_{ss}$ of the induced representation $i^G_{P}(\sigma)$.
    Using \cref{lemma:standard-module-equivalence} and \cref{construction:space-infinitesimal-character}, we can endow the set $\Theta_{\disc}(G)$ with the countable union of complex algebraic varieties and we can also define the notion of connected component.
    We omit the details.
\end{remark}

\begin{construction}
     We define the map 
     \begin{align*}
      \mathrm{inf}_G \colon \Theta_{\disc}(G) \to \Theta(G)
     \end{align*}
     as follows.

     Let $\{M, \sigma\}_G$ be a discrete pair of $G$.
     \begin{itemize}
    \item 
     If $F = \R$, the map is just taking the infinitesimal character of $i^G_{P}(\sigma)$ for any $P \in \cP^G(M)$.
     \item
     If $F$ is a $p$-adic field, we can embed $\sigma$ in $i^{M}_{Q}(\tau)$ for some cuspidal datum $(L, \tau)\in \Theta(M)$ and $Q \in \cP^{M}(L)$.
     Then, we set $\inf_G(\{M, \sigma\}_G) = [L, \tau]_G$.
     This map is well-defined by \cref{theorem:subrepresentation-theorem} and \cref{lemma:standard-module-equivalence}.
    \end{itemize}
\end{construction}

The following lemma is obvious.
\begin{lemma}\label{lemma:properness}
    The map 
    \begin{align*}
        \mathrm{inf}_G \colon \Theta_{\disc}(G) \to \Theta(G)
    \end{align*}
    taking the infinitesimal characters of representations is continuous and proper.
    Also, each fiber is a finite set.
\end{lemma}

\begin{proof}
  See \cite{NelsonVenkatesh2021-orbitmethod}*{p.157, line 1} for continuity.
  We will prove the properness. 
  It suffices to show that the inverse image of any compact subset in $\Theta(G)$ is also compact. 
  We first assume that $F$ is $p$-adic.
  Let $C \subset \Theta(G)$ be a compact set.
  Then, there exists a compact open subgroup $K$ of $G$ such that for any datum $\theta \in C$, the corresponding induced representation is generated by $K$-fixed subgroup. 
  Then, the same statement holds for the element $\pi$ in $\mathrm{inf}_G^{-1}(C)$ by \cite{Renard2010-p-adiques}*{VI 6.3, Corollaire}.
  This, \cite{NelsonVenkatesh2021-orbitmethod}*{p.157, line 1} and \cite{Waldspurger2003}*{Th\'eor\`eme VIII.1.2} imply the result. 
  In case $F = \R$, it follows easily from the classification of discrete series representations by Harish-Chandra.
  
  We prove the second statement.
  Note that the set of $G$-conjugacy classes of Levi subgroups of $G$ is finite.
  Thus, the second statement follows from \cite{Wallach1988realreductiveI}*{Theorem 5.5.6} if $F=\R$, and \cite{Renard2010-p-adiques}*{VI 6.2, Lemme} if $F$ is a $p$-adic field and the fact that the parabolic induction of any irreducible representation is of finite length.
\end{proof}

\begin{definition}
    Let $\Theta_{\disc}(G)_{\temp}$ denote the closed subset of $\Theta_{\disc}(G)$ which consists of classes $\{M, \sigma\}_G$ with $\sigma$ unitary discrete series.
    This set also endows with the structure of the countable union of algebraic varieties over $\R$, as the countable union of quotients of compact tori or $\R$-vector spaces by finite groups.
\end{definition}

\begin{remark}
  We can identify the element $\{M, \sigma\}_G \in \Theta_{\disc}(G)_{\temp}$ with the semisimple representation $i^G_P(\sigma)$.
  As the spaces $\Theta_{\disc}(G), \Theta(G)$ and $\Theta_{\disc}(G)_{\temp}$ are countable unions of quotients of locally Euclidean spaces by  finite groups, we have the notion of smooth functions, holomorphic functions on these spaces and so on.
\end{remark}

The following is well known, for example, see \cite{Wal03}*{Proposition III 4.1} and \cite{Harish-Chandra1972-Eisensteinintegral}*{Lemma 12}.
\begin{lemma}\label{lemma:tempered-conjugacy}
   Let $\{M, \sigma \}_G$ and $\{M', \sigma'\}_G$ be elements of $\Theta_{\disc}(G)_{\temp}$.
   Then, if the parabolic inductions $i^G_{P}(\sigma)$ and $i^G_{P'}(\sigma')$ have a common irreducible constituent, then we have the equality $\{M, \sigma \}_G = \{M', \sigma'\}_G$.
\end{lemma}

\begin{remark}
From the previous lemma, it is easy to see that we have a map 
    \begin{align*}
        \mathrm{StdMod}_G \to \Theta_{\disc}(G)
    \end{align*}
    induced by the map 
    \begin{align*}
        (P = MN, \sigma, \nu) \mapsto (M_1, \sigma_{1} \otimes \chi_{\nu}),
    \end{align*}
    where $M_1$ a Levi component of a parabolic subgroup $P_1$ in $M$, $\sigma_1$ is a discrete series representation of $M_1$, and $\sigma \subset i^{M}_{P_1}(\sigma_1)$ is a irreducible component.
    We can see that 
    \begin{align*}
        i^{G}_{P}(\sigma \otimes \chi_{\nu}) \hookrightarrow i^{G}_{M_1N} (\sigma_1 \otimes \chi_{\nu}).
    \end{align*}
\end{remark}

\section{On the generic irreducibility theorem}

As noted in \cite{Shi12}, the proof of the generic irreducibility theorem \cite{Sau97}*{Th\'eor\`eme 3.2} appears to contain an error. 
Konno \cite{Konno2003}*{Theorem 4.2} provided a remedy.
However, a stronger hypothesis than \cite{Sau97}*{Th\'eor\`eme 3.2} was used in \cite{Konno2003}*{p.403, l.24}.
We will show that this strengthening is unnecessary by proving the following lemma.
Furthermore, the proof of \cite{Sau97}*{Corollaire 3.3} seems to require the original version of the theorem.

\begin{lemma}
    Let $V$ be a real vector space. 
    Let $S$ be a finite subset in $V$.
    Let $T^{\vee}$ be a finite subset in the linear dual $V^{\vee}$ of $V$.
    If no element of $T^{\vee}$ vanishes at an element in $S$, then there exist $v \in \mathrm{span}_{\R}(S)$ such that $\lambda(v) \neq 0$ for all $\lambda \in T^{\vee}$.
\end{lemma}

\begin{proof}
    Let $W = \mathrm{span}_{\R}(S)$.
    Then, every element of $T^{\vee}$ has non-zero restriction on $W$.
    Using the standard hyperplane avoidance technique, we can take $v \in W$ as in the statement.
\end{proof}

Let $P=MN$ be a parabolic subgroup of $G$ with a Levi decomposition $P=M N$.
Let $\Phi(G, A_M)$ be the set of weights of the action of $A_M$ on $\mathrm{Lie}(G)$.

\begin{theorem}\label{theorem-generic-irreducibility}
    For every irreducible representation $\sigma$ of $M$, there exists a neighborhood $\cU$ of $0 \in (\fa_M)^*_{\C}$ such that $i_P^G(\sigma \otimes \chi_{\lambda})$ is irreducible for any element $\lambda \in \cU$ such that $\alpha^{\vee}(\lambda) \neq 0$, $\alpha \in \Phi(G, A_M)$.
\end{theorem}

\begin{proof}
    If $F=\R$, this follows from \cite{Renard2024}*{Theorem 1.1}.
    We assume that $F$ is a $p$-adic field.
    Then, the proof is the same as that in \cite{Konno2003}*{Theorem 4.2}, if we combine the argument at \cite{Konno2003}*{p.403, l.24} with the previous lemma.
\end{proof}

\section{The unitary dual and the Plancherel measure}

We will review some definitions and results on the unitary dual and the Plancherel measures for reductive groups over local fields.

\begin{definition}
    We write $\Irr_{\unit}(G)$ for the unitary dual of $G$.
    This space is naturally endowed with the Fell topology.
    Let $\Irr_{\temp}(G)$ denote the closed subset of $\Irr_{\unit}(G)$ consisting of tempered representations.
\end{definition}

We fix a Haar measure on $G$ throughout.
Also, the space $\Theta_{\disc}(G)$ is a countable union of quotients of locally Euclidean spaces by finite groups.
Thus, we have the quotient measure on $\Theta_{\disc}(G)$, induced from the locally Euclidean measures.
This measure naturally restricts to the subspace $\Theta_{\disc}(G)_{\temp}$.
Let $\rd\pi$ denote this measure, where $\pi = i^G_P(\sigma) \in \Theta_{\disc}(G)_{\temp}$.

\begin{theorem}[\cite{Wal03}*{Th\'eor\`eme VIII 1.1}, \cite{Wallach1992RRGII}*{Theorem 13.4}]
    There exists a positive smooth function $\mu^G(\pi)$ of moderate growth on $\Theta_{\disc}(G)_{\temp}$ such that we have 
    \begin{align*}
        f(1) = \int_{\pi \in \Theta_{\disc}(G)_{\temp}} \tr\pi (f) \mu^G(\pi) \rd\pi
    \end{align*}
    for any smooth function $f \in \Cc(G)$.
\end{theorem}

\begin{definition}\label{definition:Fourier-transform}
    We call the measure $\mu^G(\pi) \rd\pi$ the Plancherel measure of $G$ and we denote it by $\mu^G$.
    Let $\widehat{f}$ denote the function $\pi \mapsto \tr \pi(f)$ on $\Irr_{\unit}(G)$ or $\Theta_{\disc}(G)$.
    We refer to it as the Fourier transform of $f \in \Cc(G)$.
    We denote the set of Fourier transforms for $K$-finite elements in $\Cc(G)$ by $\mathcal{FT}(G)$.
\end{definition}

\begin{definition}
    Let $f \in \Cc(G)$ be a smooth function.
    We set 
    \begin{align*}
        f^*(g) = \overline{f(g^{-1})}.
    \end{align*}
\end{definition}

\begin{lemma}\label{lemma:stable-under-complex-conjugate}
    Let $\pi$ be a unitary admissible  representation of $G$ of finite length. 
    Let $f \in \Cc(G)$ be a smooth $K$-finite function.
    Then, we have 
    \begin{align*}
        \tr\pi(f^*) = \overline{\tr\pi(f)}.
    \end{align*}
\end{lemma}

\begin{proof}
    We fix an orthonormal basis $\cB$ of $V$.
    Then, we have 
    \begin{align*}
        \tr\pi(f^*) 
        &= 
        \sum_{v \in \cB} (\pi(f^{*})v, v) \\
        &=  
        \sum_{v \in \cB} \int_{G} \overline{f(g^{-1})} (\pi(g)v, v) \rd g \\
        &= 
        \sum_{v \in \cB} \int_{G} \overline{f(g) (\pi(g)v, v) } \rd g \\
        &= \overline{\tr \pi(f)}.
    \end{align*}
    Note that by admissibility, the above  sums are finite sums.
    Hence the result.
\end{proof}

\begin{lemma}\label{lemma:identification}
     The subset of $\Theta_{\disc}(G)_{\temp}$ which consists of parabolic  inductions of regular discrete series representations in the sense of \cite{Sau97}*{p.172, line 30} is a dense open subset in $\Theta_{\disc}(G)_{\temp}$ which is conull with respect to $\rd\pi$.
     Thus, we can identify $\Theta_{\disc}(G)_{\temp}$ with $\Irr_{\temp}(G)$ up to a null set with respect to $\rd\pi$.
  
    On this subset, we have an equality of induced topologies from the spaces $\Theta_{\disc}(G)_{\temp}$ and $\Irr_{\temp}(G)$. 
\end{lemma}

\begin{proof}
      The first statement follows from \cref{theorem-generic-irreducibility}.
      The second statement follows from the argument in \cite{Beuzart-Plessis2021-Plancherel-GLnE-GLnF}*{p.245, line 33}.
\end{proof}

\begin{definition}
     By the previous lemma, the measure $\mu^G$ defines a Borel measure on $\Irr_{\temp}(G)$ and also on $\Irr_{\unit}(G)$.
     The measures are also denoted by $\mu^G$.
\end{definition}

\section{A variant of the Stone-Weierstrass theorem}

In this section, we prove the main approximation theorem \cref{theorem:main-theorem} which will be used to complete the proof of the Sauvageot density principle.

Before proving the approximation theorem, we recall the gap pointed out in \cite{NelsonVenkatesh2021-orbitmethod}*{p.159, footnote}.

\begin{remark}\label{remark:counterexample}
    In \cite{Sau97}*{p.181, line 4}, for any function $f$ as \cite{Sau97}*{p.180, (1)} and any positive number $\epsilon > 0$, it is stated that there exist functions $g, h \in \cA(G)$ such that 
    \begin{align*}
        \abs{f - g \circ \mathrm{inf}_G} 
        &\leq h \circ \mathrm{inf}_G, \\
        \mu^G(h \circ \mathrm{inf}_G) 
        &< \epsilon.
    \end{align*}
      
    We assume that $G$ has a pair of discrete series representations $\pi_1 \neq \pi_2$ such that $\mathrm{inf}_G(\pi_1) = \inf_{G}(\pi_2)$. 
    Such a pair exists if $G = G_2$, see \cite{GS23theta}*{Proposition 3.2, (ii)}.
    We take the characteristic function of $\{\pi_1\}$ as $f$ and this satisfies the condition $(1)$ of \cite{Sau97}*{p.180}.
    If there exist functions $g, h \in \cA(G)$ such that 
    \begin{align*}
        \abs{f - g \circ \mathrm{inf}_G} 
        &\leq h \circ \mathrm{inf}_G, \\
        \mu^G(h \circ \mathrm{inf}_G) 
        &< \epsilon,
    \end{align*}
    then, we have 
    \begin{align*}
        \abs{1 - g \circ \mathrm{inf}_G(\pi_1)} &\leq h \circ \mathrm{inf}_G(\pi_1), \\
        \abs{g \circ \mathrm{inf}_G(\pi_2)} &\leq h \circ \mathrm{inf}_G(\pi_2).
    \end{align*}
    This implies that we have $h \circ \inf_G(\pi_1) \geq \frac{1}{2}$.
    Finally, we obtain $\mu^G(h \circ \mu^G) \geq \frac{1}{2} d(\pi_1)$, where $d(\pi_1)$ is the formal degree of $\pi_1$ and this gives a contradiction if $\epsilon < \frac{1}{2} d(\pi_1)$.
\end{remark}

Thus, to approximate the functions in \cite{Sau97}*{p.180, (1)} in the previous sense, we have to utilize more functions. 

\subsection{Approximation on locally compact Hausdorff spaces}

Let $X, Y$ be locally compact Hausdorff spaces and $p \colon X \to Y$ be a proper map.
Let $X^{\infty}=X \coprod \{\infty_{X}\}$ denote the one-point compactification of $X$.
Then, we can extend $p$ to a continuous map $p^{\infty} \colon X^{\infty} \to Y^{\infty}$.
We have the identification of the function spaces:
\[
  C_0(X) \overset{\cong}{\to} \{f \in C(X^{\infty}) \st f(\infty_{X})=0 \} \colon f \mapsto f
\]
by the extension $f(\infty_X) = 0$.
This isomorphism preserves the norms $\|\cdot\|_{\infty}$.
The space $C_0(X)$ has the norm, so it defines a metric on this space.
Also we can take the distance 
\[
d_{X}(f, S) = \inf_{g \in S}{d_{X}(f, g)}
\] 
for any element $f \in C_0(X)$ and subset $S \subset C_0(X)$.

We have the Stone-Weierstrass theorem for subalgebras of $C_0(Y)$, for example, see \cite{Rud91}*{p.122}.
We utilize the following module version of the Stone-Weierstrass theorem, which is a special case of theorems proved by Nachbin and Prolla.
We give a proof of the result to make this paper as self-contained as possible.

\begin{lemma}[c.f. \cite{Nachbin1965}*{p.54, Theorem 1}, \cite{Prolla1994}*{Theorem 9}]\label{lemma:module-version-Stone-Weierstrass}
   Let $A$ be a dense $\C$-subalgebra of $C_0(Y)$.
   Let $p^{*}A$ be a $\C$-subalgebra of $C_0(X)$ which consists of the pullback of elements in $A$ by $p$.
   Let $M$ be a subspace of $C_0(X)$ which is a $p^*A$-module with respect to pointwise multiplication.
   Then, for any $f \in C_0(X)$ we have
   \[
     d_X(f, M) = \sup_{y \in Y}{d_{p^{-1}(y)}(f|p^{-1}(y), M|p^{-1}(y))},
   \]
   where  $M|p^{-1}(y)$ denotes the set of functions obtained by restriction of elements in $M$ to the compact set $p^{-1}(y)$ and the distance in the right hand side is taken in $C(p^{-1}(y))$.
\end{lemma}

\begin{proof}
    By one-point compactification, we can assume that $X$ and $Y$ are compact and the space $A$ and $M$ contain constant functions.
    Also, by replacing $Y$ by the image of $X$, we can assume that $p$ is surjective.
    
    We obviously have the inequality
    \[
      d_X(f, M) \geq \sup_{y \in Y}{d_{p^{-1}(y)}(f|p^{-1}(y), M|p^{-1}(y))},
    \]
    so we will prove the converse.
    Let $D$ denote the right hand side of this inequality and take $\epsilon >0$.
    It suffices to show that left hand side is less than $D+\epsilon$.
    
    By the definition of supremum, for each $y \in Y$, there exists an element $g_y \in M$ such that
    \[
     d_{p^{-1}(y)}(f|p^{-1}(y), g_y|p^{-1}(y)) < D+\epsilon.
    \]
    By a simple argument using $p$ is a closed map, we can take an open neighborhood $y \in V_y$ such that on $U_y=p^{-1}(V_y)$ we have 
    \[
      d_{U_y}(f|U_y, g_y|U_y) < D+\epsilon.
    \]
    Since the open subsets $V_y$ cover the space $Y$, we can take a finite open subcover $\{ V_i=V_{y_i} \, | \, i=1,2, ..., n \}$.
    We can also take a partition of unity $\rho_{i}$ associated with this cover.

    Set $g_i=g_{y_i}$ and
    \[
       g = \sum_{i} p^{*}(\rho_{i}) g_{i} \in C_0(X).
    \]
    Then, we have
    \begin{align*}
       \lvert f(x)-g(x) \rvert 
       &= \lvert \sum_{i} \rho_i(p(x))(f(x)-g_{i}(x)) \rvert \\
       &\leq \sum_{i} \rho_i(p(x)) \lvert f(x)-g_i(x) \rvert \\
       &< (\sum_{i} \rho_i)(p(x)) (D+\epsilon) \\
       &= D+\epsilon.
    \end{align*}
    By usual Stone--Weierstrass' theorem, we can approximate each $\rho_i$ by an element $\sigma_i$ of $A$.
    As $M$ is a $p^{*}A$-module by multiplication, the function
    \[
       h = \sum_{i} p^{*}(\sigma_{i}) g_{i} 
    \]
    belongs to $M$ and approximates $g$.
    Hence the result.
\end{proof}

\begin{remark}
    The results by Nachbin and Prolla can be applied to more general settings, see \cite{Nachbin1965}*{p.54, Theorem 1} and \cite{Prolla1994}*{Definition 2, Theorem 9}.
\end{remark}

\subsection{Approximation on measure spaces}

Let $X, Y$ be locally compact Hausdorff spaces and $p \colon X \to Y$ be a proper map.
Let $\mu_X$ be a Radon measure on $X$ and we denote the pushforward measure $p_*\mu_X$ by $\mu_Y$.
Then, $\mu_Y$ is also a Radon measure on $Y$.

\begin{definition}
   Let $(Z, \mu)$ be a topological space with a positive Borel measure.
   A function $\phi \colon Z \to \C$ is said to be $\mu$-Riemann integrable if it is bounded, compactly supported and its set of discontinuous points is $\mu$-null set. 
   We denote the set of $\mu$-Riemann integrable functions by $\mathcal{RI}(Z, \mu)$.
   We assume that $Z$ is a locally compact Hausdorff space and $\mu$ is a positive Radon measure on $Z$.
   Let $\cL^1(Z, \mu)$ denote the set of $\mu$-integrable functions on $Z$.
\end{definition}

\begin{example}
    The characteristic function $\mathbf{1}_{U}$ of an open set $U$ is $\mu$-Riemann integrable if and only if $U$ is relatively compact and its boundary $\partial U$ is a $\mu$-null set.
\end{example}

\begin{definition}\label{definition:topology}
     Let $M$ be a subspace of $\cL^1(X, \mu_X)$ such that any element $m \in M$ is $\mu_X$-integrable and stable under pointwise complex conjugation.
     Let $\tau_M$ denote the linear topology on $\cL^1(X, \mu_X)$ whose fundamental system of neighborhoods of $0$ is generated by 
     \begin{multline*}
        U_{\epsilon} = 
        \{ h \in \cL^1(X, \mu_X) \st \\ \text{There exists $m \in M$ such that } \abs{h} \leq m, \mu_X(m) < \epsilon  \},
     \end{multline*}
     for $\epsilon >0$.
\end{definition}

\begin{lemma}\label{lemma:closure-invariant}
    Let $\overline{M}^{\tau_M}$ denote the closure of $M$ with respect to $\tau_M$.
    Then, we have $\tau_M = \tau_{\overline{M}^{\tau_M}}$.
\end{lemma}

\begin{proof}
    As $M$ is stable under complex conjugation, it suffices to show the lemma after replacing $\cL^1(X, \mu_X)$ with the subspace $\cL^1(X, \mu_X, \R)$ of real valued functions and $M$ with the subspace $M_{\R}$ of real valued functions in $M$.
    If we consider the topology defined by $\overline{M}^{\tau_M}$, we write the open subset $U_{\epsilon}$ as $U'_{\epsilon}$, i.e. 
    \begin{multline*}
        U'_{\epsilon} 
        = 
        \{ h \in \cL^1(X, \mu_X) \st \\ \text{There exists $m' \in \overline{M}^{\tau_M}$ such that } \abs{h} \leq m', \mu_X(m') < \epsilon  \}.
    \end{multline*}
    Then, it is trivial that we have $U_{\epsilon} \subset U'_{\epsilon}$.
    We claim that if $\epsilon' < \epsilon$, then we have $U'_{\epsilon'}\subset U_{\epsilon}$.
    We assume that a function $h \in \cL^1(X, \mu_X)$ is in $U'_{\epsilon'}$.
    Then, we can find $m' \in \overline{M}^{\tau_M}$ and we have 
    \begin{align*}
        &\abs{h} \leq m', \\
        &\mu_X(m') < \epsilon'.
    \end{align*}
    We set $\delta = \epsilon - \epsilon'$. 
    Then, we have $\delta > 0$.
    As the function $m'$ is in $\overline{M}^{\tau_M}$, there exist functions $n_1 \in M$ and $n_2 \in M$ such that $\abs{m' - n_1} \leq n_2$ and $\mu_X(n_2) < \frac{1}{2}\delta$.
    Then, we have 
    \begin{align*}
        \abs{h} \leq n_1 + n_2
    \end{align*}
    and 
    \begin{align*}
        \mu_X(n_1 + n_2) < \epsilon' + 2 \times \frac{1}{2} \delta = \epsilon.
    \end{align*}
    Thus, we have $U_{\epsilon'}' \subset U_{\epsilon}$.
\end{proof}

\begin{corollary}\label{corollary:double-closure}
   We have $\overline{\overline{M}^{\tau_M}}^{\cT_{\overline{M}^{\tau_M}}} = \overline{M}^{\tau_M}$.
\end{corollary}

The following lemma is \cite{Sau97}*{Lemme 2.3}.
\begin{lemma}\label{lemma:closure-module}
    Let $A \subset C_0(Y)$ be a dense $\C$-subalgebra.
    Let $M \subset \cL^1(X, \mu_X)$ be a $p^*A$-module with respect to pointwise multiplication and stable under complex conjugation, such that any element $m \in M$ is $\mu_X$-integrable.
    We assume that $M$ also satisfies the following condition:
    For any function $m \in M$, there exists a function $n \in M$ such that $\abs{m} \leq n$.

    Then, for any function $m \in M$ and $c \in C_0(Y)$, the function $(p^*c) m$ is in $\overline{M}^{\tau_M}$.
\end{lemma}

\begin{proof}
    We take a function $n \in M$ as in the hypothesis of the lemma. 
    There exists a function $a \in A$ such that $\abs{c-a} \leq \epsilon$ with $\epsilon$ sufficiently small with respect to $\mu_X(n)$.
    Then, we have 
    \begin{align*}
        \abs{(p^*c)m - (p^*a)m} \leq \epsilon n.
    \end{align*}
    Hence the result.
\end{proof}

\begin{remark}\label{remark:non-Hausdorff}  
   Note that the proof of \cref{lemma:closure-module} and \cref{corollary:double-closure} does not use the Hausdorff property of $X$.
\end{remark}

\begin{theorem}\label{theorem:main-theorem}
  Let $A$ be a dense $\C$-subalgebra of $C_0(Y)$.
  Let $M \subset C_0(X) \cap \cL^1(X, \mu_X)$ be a $p^*A$-module with respect to pointwise multiplication.
  We assume the following conditions:
   \begin{enumerate}
   \item $M$ is stable under pointwise complex conjugation.
   \item The restriction of the elements of $M$ to $p^{-1}(y)$ is dense in $C(p^{-1}(y))$ for every $y \in Y$.
   \item For any function $m \in M$, there exists a function $n \in M$ such that $\abs{m} \leq n$.
   \end{enumerate}
   Then, we have $\mathcal{RI}(X, \mu_X) \subset \overline{M}^{\tau_M}$.
\end{theorem}

\begin{proof}
   It suffices to show that we have $\mathcal{RI}(X, \mu_X) \subset \overline{\overline{M}^{\tau_M}}^{\cT_{\overline{M}^{\tau_M}}}$ by \cref{corollary:double-closure}.
   We will write down this condition.
   Take any function $\phi \in \mathcal{RI}(X, \mu_X)$ and any positive number $\epsilon > 0$.
   We want to show that there exist functions $m'_1, m'_2 \in \overline{M}^{\tau_M}$ such that we have 
   \begin{align*}
       \abs{\phi - m'_1} &\leq m'_2, \\
       \mu_X(m'_2) &< \epsilon.
   \end{align*}
  
   We fix $\epsilon > 0$.
   By assumption, the support of $\phi$ is compact. 
   Thus, we can take a positive function $c \in C_c(Y)$ such that on the support of $\phi$, we have $p^*c  = 1$.
   By the assumptions $(2), (3)$ for $M$, it is easy to find a positive function $m_c \in M$ such that we have $p^*c \leq m_c$ on $X$. 
   Indeed, we claim that for any point $x \in X$, we can find $m_x \in M$ such that $m_x(x) > 0$.
   This claim clearly proves the existence of the function $m_c$.
   By the assumption $(2)$, we can take a function $n_x \in M$ such that we have $n_x(x) \neq 0$.
   Then, the assumption $(3)$ implies that there exists $m_x \in M$ such that we have $\abs{n_x} \leq m_x$.
   We have $m_x(x) \geq \abs{n_x(x)} > 0$ and this proves the claim.

   We take a positive number $\epsilon_1 > 0$ such that 
   \begin{align*}
         \epsilon_1 < \frac{1}{3\|m_c\|_{\infty} + \mu_X(m_c)} \epsilon.
   \end{align*}
   By \cite{Bou04}*{IV 12 Lemma 5}, we obtain functions $c_1, c_2 \in C_c(X)$ such that 
     \begin{align*}
         \abs{\phi - c_1} &\leq c_2,    \\
         \mu_X(c_2) &\leq \epsilon_1.
     \end{align*}
    We apply \cref{lemma:module-version-Stone-Weierstrass} to obtain functions $m_1, m_2 \in M$ such that 
    \begin{align*}
        \abs{c_1 - m_1} &< \epsilon_1, \\
        \abs{c_2 - m_2} &< \epsilon_1.
    \end{align*}
    At this stage, we have 
    \begin{align*}
        \abs{\phi - (p^*c)m_1} 
        &\leq (p^*c)\abs{\phi-c_1} + (p^*c)\abs{c_1-m_1} \\
        &\leq (p^*c)m_2 + (p^*c)\abs{c_1 - m_1} + (p^*c)\abs{c_2 - m_2} \\
        &\leq (p^*c)m_2 + \epsilon_1 m_c + \epsilon_1 m_c \\
        &= (p^*c)m_2 + 2 \epsilon_1 m_c.
    \end{align*}
    Here, we have 
    \begin{align*}
        \mu_X((p^*c)m_2) 
        &\leq \mu_X((p^*c)\abs{m_2-c_2}) + \mu_X((p^*c)c_2) \\
        &\leq \epsilon_1 \mu_X(p^*c) + \epsilon_1 \| c \|_{\infty}.
    \end{align*}
    Thus, we have 
    \begin{align*}
          \mu_X((p^*c)m_2 + 2 \epsilon_1 m_c) \leq (3\|m_c\|_{\infty} + \mu_X(m_c)) \epsilon_1 < \epsilon
    \end{align*}
    and $(p^*c)m_1, (p^*c)m_2 + 2 \epsilon_1 m_c \in \overline{M}^{\tau_M}$ by \cref{lemma:closure-module}.
    Hence the result.
\end{proof}

\begin{example}\label{example:main-example}
    We take
    \begin{itemize}
      \item the space $\Theta_{\disc}(G)_{\temp}$ as $X$,
      \item the space $\Theta(G)_{\herm}$ as $Y$,
      \item the map $\inf_G$ as $p$,
      \item the Radon measure $\mu^G$ as $\mu_X$,
      \item the algebra $\cA(G)$ as $A$ and 
      \item the module $\mathcal{FT}(G)$ as $M$, see \cref{definition:Fourier-transform}.
    \end{itemize}
    The space $\mathcal{FT}(G)$ is an $\cA(G)$-module by the Paley-Wiener theorem \cite{ClozelDelorme1990}*{Th\'eor\`eme 1} if $F=\R$ and \cite{Renard2010-p-adiques}*{VI.10.3 Th\'eor\`eme} if $F$ is a $p$-adic field.
    Then, Lemmas in \cite{Sau97}*{Lemme 3.1, 3.4, 3.5}, \cref{lemma:stable-under-complex-conjugate-algebra}, \cref{lemma:properness}, \cref{lemma:tempered-conjugacy}, \cref{lemma:stable-under-complex-conjugate}, and the linear independence of irreducible characters imply that this sextuple satisfies the condition of \cref{theorem:main-theorem}.
    Also, the similar result holds for any reductive group schemes over a product of local fields.
\end{example}

\begin{definition}
    Let $\Irr_{\unit}(G)_{\mathrm{good}}$ denote the subset of $\Irr_{\unit}(G)$ which consists of the parabolic inductions associated to a discrete pair $(M, \sigma)$ with $\sigma$ unitary and regular in the sense of \cite{Sau97}*{p.172, line 30}.
    Let $\Irr_{\unit}(G)_{\mathrm{bad}}$ denote the complement in $\Irr_{\unit}(G)$ of $\Irr_{\unit}(G)_{\mathrm{good}}$.
\end{definition}

\begin{remark}
    We have $\mu^G(\Irr_{\unit}(G)_{\mathrm{bad}}) = 0$ by \cref{theorem-generic-irreducibility}.
\end{remark}

\begin{corollary}\label{corollary:density-parabolically-induced}
    Let $G$ be a reductive group scheme over a product of local fields of characteristic $0$.
    Let $\phi$ be a function in $\mathcal{RI}(\Irr_{\temp}(G), \mu^G)$ and $\epsilon$ be a positive real number.
    Then, there exist functions $h_1, h_2 \in \Cc(G)$ such that we have 
    \begin{align*}
         \abs{\phi - \widehat{h_1}} &\leq \widehat{h_2}, \\
         \mu^G(\widehat{h_2}) & < \epsilon,
    \end{align*} 
    on the set $\Irr_{\unit}(G)_{\mathrm{good}}$.
\end{corollary}

\begin{proof}
    We take a function $\phi \in \mathcal{RI}(\Irr_{\temp}(G), \mu^G)$.
    Then, by \cref{theorem:main-theorem} and \cref{example:main-example}, we can find functions $h_1, h_2 \in \Cc(G)$ such that we have 
    \begin{align*}
        \abs{\phi - \widehat{h_1}} &\leq \widehat{h_2}, \\
        \mu^G(\widehat{h_2}) &< \epsilon,
    \end{align*}
    on the subset ${\Irr_{\unit}(G)}_{\mathrm{good}}$ in $\Irr_{\unit}(G)$, through the identification explained in \cref{lemma:identification}.
    Hence the result.
\end{proof}

\section{The Sauvageot density principle}

In this section, we will prove the following result. 

\begin{theorem}[\cite{Sau97}*{Th\'eor\`eme 5.4}]\label{theorem:construction}
    Let $G$ be a reductive group schemes over a product of local fields.
    Let $C$ be a compact subset of $\Theta(G)$ and let $\epsilon > 0$ be a positive real number.
    Then, there exists a function $h \in \Cc(G)$ which satisfies the following conditions.
    \begin{enumerate}
      \item    
      We have $\widehat{h} \geq 0$ on $\Irr_{\unit}(G)$.
      \item    
      We have $\mu^G(\widehat{h}) \leq \epsilon$.
      \item     
      For any $\pi \in \Irr_{\unit}(G)_{\mathrm{bad}}$ with $\inf_G(\pi) \in C$, we have
      \begin{align*}
           \widehat{h}(\pi) \geq 1.
      \end{align*}  
    \end{enumerate}
\end{theorem}

Before proving this theorem, we first show that this theorem and \cref{corollary:density-parabolically-induced} imply the following Sauvageot density principle.

\begin{theorem}[Sauvageot's density principle, \cite{Sau97}*{Th\'eor\`eme 7.3}]\label{theorem:density-principle}
     Let $\phi$ be a bounded, Borel measurable, and compactly supported function on $\Irr_{\temp}(G) \subset \Irr_{\unit}(G)$ such that the set of discontinuity of $\phi$ is a $\mu^G$-null set.
     Let $\epsilon > 0$ be a positive real number.
     Then, there exist functions $h_1, h_2 \in \Cc(G)$ such that we have 
     \begin{align*}
      \abs{\phi - \widehat{h_1}} &\leq \widehat{h_2}, \\
      \mu^G(\widehat{h_2}) & < \epsilon
     \end{align*} 
     on $\Irr_{\unit}(G)$.
\end{theorem}

\begin{proof}[Proof that \cref{theorem:construction} implies \cref{theorem:density-principle}]
    We take a function $c \in C_c(\Theta(G))$ such that $c' = \inf_G^{*} c$ is equal to $1$ on the support of $\phi$. 
    We take a positive number $\epsilon_1$ such that 
    \begin{align*}
          \epsilon_1 < \frac{1}{2(\|c\|_{\infty} + 1)} \epsilon.
    \end{align*}
    We take the functions $h'_1, h'_2 \in \Cc(G)$ satisfying the condition of \cref{corollary:density-parabolically-induced} for $\phi$ and $\epsilon_1$.
    Then, we have 
    \begin{align*}
      \abs{\phi - \widehat{h'_1}} &\leq \widehat{h'_2}, \\
      \mu^G(\widehat{h'_2}) & < \epsilon_1
    \end{align*} 
    on the complement of $\Irr_{\unit}(G)_{\mathrm{bad}}$.
    Then, we have 
    \begin{align*}
      \abs{\phi - c'\widehat{h'_1}} \leq c'\widehat{h'_2}.
    \end{align*}
    By \cref{theorem:construction}, we can take an element $h'_3 \in \Cc(G)$ such that we have 
    \begin{align*}
      \abs{\phi - c'\widehat{h'_1}} + c'\widehat{h'_2} &\leq \widehat{h'_3}, \\
      \mu^G(\widehat{h'_3}) &< \epsilon_1
    \end{align*}
    on the set $\Irr_{\unit}(G)_{\mathrm{bad}}$.
    Then, we have 
    \begin{align*}
      \abs{\phi - c'\widehat{h'_1}} &\leq c'\widehat{h'_2} + \widehat{h'_3}, \\
      \mu^G(c'\widehat{h'_2} + \widehat{h'_3}) &< (\|c\|_{\infty} + 1) \epsilon_1 < \frac{1}{2} \epsilon.
    \end{align*}
    By applying \cref{corollary:double-closure}, \cref{lemma:closure-module} and \cref{remark:non-Hausdorff}, we can find functions $h_1, h_2 \in \Cc(G)$ such that we have 
    \begin{align*}
      \abs{\phi - \widehat{h_1}} 
      &\leq \widehat{h_2}, \\ 
      \mu^G(\widehat{h_2}) 
      &< \epsilon.
    \end{align*}
    Hence the result.
\end{proof}

\begin{lemma}\cite{Sau97}*{Corollaire 6.2}
    \cref{theorem:construction} for reductive groups over local fields implies the general case of \cref{theorem:construction}.
\end{lemma}

Thus, we return to the setting where $G$ is a reductive group over a local field $F$.

\begin{lemma}[c.f. \cite{Sau97}*{Lemme 2.6}]\label{lemma:Besikovitch}
   Let $V$ be a finite dimensional  $\R$-vector space and let $\Gamma$ be a discrete subgroup of $V$. 
   We set $X = V/\Gamma$.
   For any compact set $C$ of $X$ and a family of open neighborhood $U_{x}$ for each point $x \in C$, there exists a finite covering $\{ W_i \}_{i \in I}$ of $C$ by relatively compact open subsets $W_i$ which satisfy the following property.
   \begin{enumerate}
     \item     
      For any $i \in I$, there exists a point $x \in C$ such that $\overline{W_i} \subset V_x$.
      \item     
      Each point of $\overline{W_i}$ is contained in at most $2^{\dim X}$ elements of $\{ \overline{W_j} \}_{j \in I}$.
   \end{enumerate} 
\end{lemma}

\begin{proof}
    We may assume that $V = \R^n$ and $\Gamma = \Z^m$ for $m \leq n$.
    Let $\{p_i\}_{i=1,...,n}$ be the standard coordinate functions on $\R^n$.
    We set $\cU = \{ U_x \}_{x \in K}$. 
    Let $K_{\epsilon}$ denote the closed $\epsilon$-neighborhood of $K$.
    Then, as $K$ is compact, there exists an $\epsilon > 0$ such that 
    \begin{align*}
        K \subset K' = K_{\epsilon} \subset \bigcup \cU.
    \end{align*}
    We consider the subgroup $\frac{1}{2^{k}}\Z^n / \Z^m$ for sufficiently large $k$.
    This defines a family of boxes in $X$ with width $\frac{1}{2^k}$, where  a box in $X$ with width $r$ is the projection of the sets in $\R^n$ of the form 
    \begin{align*}
        \{ y \in X \st \abs{p_i(y)-p_i(c)} < r/2, i=1,2,...,n \}
    \end{align*}
    for some $c \in X$ and $r > 0$. 
    For each element $B$ in the family above, we take a box $B'$ with width $\frac{11}{10} \frac{1}{2^k}$, roughly speaking, the thickened box $B'$ obtained from $B$.
    We will denote the family of the thickened boxes by $\cB_k$.

    We set $\cB_k(K) = \{ B \in \cB_k \st \overline{B} \cap K \neq \emptyset \}$.
    This set is obviously finite, as the set $K$ is compact.
    By enlarging $k$, we may assume that 
    \begin{align*}
        K \subset \bigcup_{B \in \cB_k(K)} B \subset \bigcup_{B \in \cB_k(K)} \overline{B} \subset K' \subset \bigcup \cU.
    \end{align*}
    As $K'$ is compact, the set $\bigcup_{B \subset \cB_k(K)} \overline{B}$ is compact in $X$.
    Thus, applying the argument using the Lebesgue number for this set and the covering induced by $\cU$, we divide  each box in $\cB_k(K)$ into $2^n$ boxes several times to obtain a new set of boxes $\cB$ in $X$, such that for any element $B \in \cB$, there exists an element $U \in \cU$ such that $\overline{B} \subset U$.
    We take the set $\cB$ as the set $\{W_i\}_{i}$ in the statement.
    Then, the first property is satisfied by definition.
    By noting that the projection $V \twoheadrightarrow V/\Gamma$ is a   local isomorphism and counting locally, it is easy to see that the second property holds true for this set.
\end{proof}

We recall the notion of vertical bands with compact real part, which are used in the proof.

\begin{definition}
    First, assume that $F=\R$.
    In this case, by a vertical band with compact real part in $\fh^* \otimes_{\R} \C$, we mean a subset $B$ of $\fh^* \otimes_{\R} \C$ with compact real part with respect to $\fh^{*}_{\R}$.
    We call the projection of a vertical band with compact real part in $\fh^* \otimes_{\R} \C$ onto $\Theta(G)$, a vertical band with compact real part in $\Theta(G)$.

    We assume that $F$ is a $p$-adic field.
    By a vertical band with compact real part in $\fa^*_{M, \C}$, we mean a subset of $\fa^*_{M, \C}$ with compact real part with respect to $X^*(M) \otimes_{\Z} \R$. 
    We call the finite union of images of vertical bands with compact real part in $\fa^*_{M, \C}$ under the map $(2.2)$, a vertical band with compact real part in $\Theta(G)$.
\end{definition}

The proof of \cite{Sau97}*{Lemme 5.1} seems  unclear to the author when $F$ is $p$-adic.
We will give a proof following the method of Delorme \cite{Del86}. 

We fix the minimal parabolic $P_0$ and we only consider the standard parabolic subgroups.
Let $\leq_{P_0}$ be the order defined by the cone generated by the positive roots (or equivalently by the simple roots) corresponding to $P_0$.

\begin{lemma}[Borel--Wallach, IV.4.13]\label{lemma:Borel-Wallach}
    Let $(P, \sigma, \nu)$ be a Langlands datum and $J(P, \sigma, \nu)$ be the corresponding Langlands quotient.
    Then, for any irreducible constituent $J(P', \sigma', \nu')$ of the standard module $i^G_P(\sigma \otimes \chi_{\nu})$ other than $J(P, \sigma, \nu)$, we have 
    \begin{align*}
        \nu' <_{P_0} \nu.
    \end{align*}
\end{lemma}

\begin{remark}
    From the proof of \cref{lemma:standard-module-equivalence}, we can see that any element $\pi \in \Theta_{\disc}(G)$ gives a set $\{ (P, \sigma_i, \nu) \}$ of Langlands data, where $P = MN$ is a standard parabolic subgroup, $\sigma_i$ is a tempered representation of $M$, and $\nu \in \fa_{M}^{*}$ be a positive regular element.
    Note that this $\nu$ is uniquely determined by $\pi$.
    This construction and the order $\leq_{P_0}$ induces a pre-order $\leq_{P_0}$ on the set $\Theta_{\disc}(G)$.
    Similarly, the construction induces a pre-order $\leq_{P_0}$ on the set of standard modules of $G$.
\end{remark}

\begin{lemma}\label{lemma:decomposition-irreducibles-to-standards}
    Let $(P, \sigma, \nu)$ be a Langlands datum and $J(P, \sigma, \nu)$ be the corresponding Langlands quotient.
    Then, we have 
    \begin{align*}
        J(P, \sigma, \nu) = i^G_P(\sigma \otimes \chi_{\nu}) + \sum_{(P', \sigma', \nu'), \nu' <_{P_0} \nu} m_{(P', \sigma', \nu')} J(P', \sigma', \nu')
    \end{align*}
    with $m_{(P', \sigma', \nu')} \in \Z$ uniformly bounded by a constant $c_G$, which only depends on $G$.
\end{lemma}

\begin{proof}
    We can fix a cuspidal datum $\theta \in \Theta(G)$ and consider the fiber of this point.
    Then, the number of irreducible representations over $\theta$ is at most $\abs{W}$.
    Also, the multiplicity of each irreducible representation in the standard module is at most $\abs{W}$.
    Thus, taking the inverse matrix of unipotent matrix with respect to the basis consisting of irreducible representations and the basis consisting of standard modules, we obtain the result.
\end{proof}

We also need the following well-known lemma.
\begin{lemma}\label{lemma:Weyl-chamber-is-deformation-retract}
    Let $(V, \Phi)$ be a root system.
    Let $C$ be a closed Weyl chamber.
    For any element $\lambda \in C$, let $\lambda^+$ denote the unique Weyl conjugate of $\lambda$ in $C$.
    Then, the map 
    \begin{align*}
        \lambda \mapsto \lambda^+
    \end{align*}
    is continuous.
    Also, if an element $\nu \in C$ is in a facet $F$, then the map 
    \begin{align*}
        \lambda \mapsto \lambda^{+}
    \end{align*}
    is realized by an element in $W$ which fixes $F$ if it is restricted to an open neighborhood of $\nu$.
\end{lemma}

\begin{proof}
    It suffices to show that it is continuous at each point in  $C$ as $C$ is a fundamental domain for the action of the Weyl group $W$ on $V$.
    Let $\alpha_1, ..., \alpha_r$ (resp. $\varpi_1, ..., \varpi_r$) be the simple roots (resp. fundamental weights) corresponding to $C$.
    An element $\lambda$ in $C$ can be written as 
    \begin{align*}
        \lambda = \sum_{i \in I} c_i \varpi_i
    \end{align*}
    with $I \subset \{ 1, ..., r\}$ and $c_i > 0$.
    This element is in a unique facet which we denote by $F_I$.
    We set $J = \{1, ..., r\} \setminus I$.
    An element close to $\lambda$ can be written as 
    \begin{align*}
        \sum_{i \in I} c'_i \varpi_i + \mu
    \end{align*}
    with $c'_i$ close to $c_i > 0$ and $\mu$ in the space spanned by $\alpha_j$ for $j \in J$.
    Then, there exists a Weyl element $w \in \langle s_{\alpha_j} \, \vert \, j \in J \rangle$ such that $w \mu$ is the Weyl chamber $C_J$ corresponding to $J$ in the space spanned by $\alpha_j$ for $j \in J$.
    Then, 
    \begin{align*}
        w\lambda = \sum_{i \in I} c'_i \varpi_i + w\mu
    \end{align*}
    is in $C$ if $c'_i$ is sufficiently close to $c_i > 0$.
    Thus, we have $\lambda^+ = w \lambda$.
    We note that the action of $w$ fixes the elements in $F_I$.
    Hence the result.
\end{proof}

\begin{lemma}\label{lemma:neighborhood-of-inverse-image}
    Let $p \colon X \to Y$ be a closed continuous map between locally compact Hausdorff spaces.
    Let $y \in Y$ be a point and let $U$ be a open neighborhood of $p^{-1}(y)$.
    Then, there exists an open neighborhood $V$ of $y$ such that we have 
    \begin{align*}
        p^{-1}(V) \subset U.
    \end{align*}
\end{lemma}

\begin{proof}
    Take $V = Y \setminus p(X \setminus U)$.
\end{proof}

The following proposition is a slight generalization of a result by Muic \cite{Muic2020-tracePaleyWiener}.

\begin{proposition}\label{proposition:explicit-construction-of-elements-of-hecke-algebra}
    Let $(M, \sigma)$ be a discrete datum and let $P \in \cP^G(M)$ be a parabolic subgroup of $G$. 
    We also take a compact open subgroup $K'$ of $G$ such as \cite{Renard2010-p-adiques}*{VI.10.6, Proposition} and the $K'$-fixed part $i^G_{P}(\sigma)^{K'}$ generates $i^G_{P}(\sigma)$.
    Let $\cO \subset \Theta_{\disc}(G)$ be the connected component containing $(M, \sigma)$.

    Let $\phi \in \cO(\cO)$ be a regular function on the connected component $\cO$.
    We regard this function as a function on the space $\Theta_{\disc}(G)$ by extension by zero.
    Then, there exists an element $h_{\phi} \in \Cc(G)$ such that, for any standard module $\sigma$, we have 
    \begin{align*}
        \tr(\sigma)(h_{\phi}) =  \frac{\dim\sigma^{K'}}{\dim\pi^{K'}} \phi(\pi)
    \end{align*}
    where the representation $\pi$ is the image of $\sigma$ by the map
    \begin{align*}
        \mathrm{StdMod}_G \to \Theta_{\disc}(G).
    \end{align*}
\end{proposition}

\begin{remark}
    By the assumption that the $K'$-fixed part $i^G_{P}(\sigma)^{K'}$ generates $i^G_{P}(\sigma)$ implies that the integer  $\dim\sigma^{K'}$ is strictly positive. 
    It is trivial that this number is bounded by $\dim\pi^{K'}$, which is constant on $\cO$.
\end{remark}

\begin{proof}
    We note that the set of standard modules of $G$ gives a basis of the Grothendieck group $R(G)$ of the representations of $G$ of finite length.
    Thus, the definition in the statement gives a linear functional $F_{\phi}$ on $R(G)$.
    By \cite{BernsteinDeligneKazhdan1986} and the fact above, it suffices to show that the following: For a parabolic subgroup $P = MN$ and a tempered representation $\sigma$ of $M$, the map 
    \begin{align*}
        \lambda \in \fa^{*}_{M, \C} \mapsto F_{\phi}(i^G_{P}(\sigma \otimes \chi_{\lambda}))
    \end{align*}
    is regular.
    We first note that the semisimplification of $i^G_{P}(\sigma \otimes \chi_{\lambda})$ does not change if we replace $P$ with another element in $\cP^G(M)$.
    We denote by $i^G_{M}(\sigma \otimes \chi_{\lambda})$ the semisimplification.
    There exists a Weyl element $w$ such that 
    \begin{enumerate}
        \item
        $w\lambda$ is positive with respect to $P_0$.
        \item 
        $w(M)$ is standard.
    \end{enumerate}
    We have
    \begin{align*}
        F_{\phi}(i^G_{M}(\sigma \otimes \chi_{\lambda})) 
         =
        F_{\phi}(i^G_{w(M)}(w\sigma \otimes \chi_{w\lambda}))
    \end{align*}
    There exists a unique standard Levi subgroup $M_1$ including $w(M)$ such that 
    \begin{align*}
        \mathrm{Re}(w\lambda) \in \fa^{*}_{M_1}
    \end{align*}
    and this element is regular in $\fa^{*}_{M_1}$.
    Then, we have 
    \begin{align*}
        i^G_{w(M)}(w\sigma \otimes \chi_{w\lambda}) 
        = 
        i^G_{M_1}(i^{M_1}_{w(M)}(w\sigma) \otimes \chi_{w\lambda}),
    \end{align*}
    here the representation $i^{M_1}_{w(M)}(w\sigma)$ is the isomorphism class of the parabolic induction $i^{M_1}_{w(P)}(w\sigma)$.

    Furthermore, there exists a standard Levi subgroup $M_2 \subset M$ such that $\sigma$ is an irreducible component of the parabolic induction of a discrete series representation $\sigma_2$ of $M_2$ to $M$.
    Note that the class $\{M_2, \sigma_2\}_G$ does not depend on the character $\lambda \in \fa^{*}_{M, \C}$.
    If this class does not belong to $\cO$, then the value 
    \begin{align*}
         F_{\phi}(i^G_{M}(\sigma \otimes \chi_{\lambda})) 
    \end{align*}
    is identically zero for all $\lambda \in \fa^{*}_{M, \C}$.

    We assume that the class is in $\cO$.
    Let $W$ be the space $i^{K}_{K \cap M}(\sigma_{\restriction{K \cap M}})$.
    Let $V$ be the space $i^{K}_{K \cap M_2}({\sigma_2}_{\restriction{K \cap M}})$.
    Then, by definition of $F_{\phi}$, the value 
    \begin{align*}
        F_{\phi}(i^G_{w(M)}(w\sigma \otimes \chi_{w\lambda}))
    \end{align*}
    is equal to 
    \begin{align*}
        \frac{\dim(W^{K'})}{\dim(V^{K'})} \phi(\{wM_2, w\sigma_2 \otimes \chi_{w(\lambda)}\}_G)
        = 
         \frac{\dim(W^{K'})}{\dim(V^{K'})} \phi(\{M_2, \sigma_2 \otimes \chi_{\lambda}\}_G)
    \end{align*}
    It is easy to see that this function is regular in $\lambda$.
    Hence the result.
\end{proof}

\begin{remark}
    We can find an element of the Hecke algebra $h_{\phi}$ if we assume that $\phi \in \cO(\Theta_{\disc}(G))$ is supported on a finite number of connected components, by taking a sum of elements of the Hecke algebra obtained from the previous proposition.
\end{remark}

\begin{lemma}[cf. \cite{Sau97}*{Lemme 5.1}]\label{lemma:key-lemma}
    Let $C \subset \Theta(G)$ be a compact set of cuspidal data, and let $\theta \in C$.
    Let $V$ be an open neighborhood of $\theta$, and let $\epsilon > 0$.
    For an irreducible representation $\pi$, let $d(\pi) \in \Theta_{\disc}(G)$ be the image of its standard module under the map in \cref{proposition:explicit-construction-of-elements-of-hecke-algebra}.

    There exist a constant $c_{G,C}$ depending only on $G$ and $C$, an open neighborhood $V_{\theta,\epsilon}$ of $\theta$, and a function $h_{\theta,\epsilon} \in \Cc(G)$ with the following properties.
    \begin{enumerate}
        \item    
        $V_{\theta, \epsilon} \subset V$. 
        \item   
        Let $\pi_1,\ldots,\pi_n$ be the points of the fiber $\mathrm{inf}_G^{-1}(\theta)$ in $\Theta_{\disc}(G)$.
        Each $\pi_i$ has an open neighborhood $U_{\pi_i}$ such that
        \begin{align*}
            \mathrm{inf}_{G}^{-1}(V_{\theta, \epsilon})
            = 
            \coprod_i U_{\pi_i}.
        \end{align*}
        \item 
        Suppose that $\pi_i$ is good.
        Then every irreducible unitary representation $\pi$ with $d(\pi) \in U_{\pi_i}$ is good and satisfies
        \begin{align*}
            0 \leq \widehat{h_{\theta,\epsilon}}(\pi) \leq \epsilon.
        \end{align*} 
        \item    
        Every bad irreducible unitary representation $\pi$ with $d(\pi) \in U_{\pi_i}$ satisfies
        \begin{align*}
            1 \leq \widehat{h_{\theta, \epsilon}}(\pi) \leq c_{G, C}.
        \end{align*} 
    \end{enumerate}
\end{lemma}

\begin{proof}
    We choose pairwise disjoint neighborhoods of the points $\pi_i$ and apply \cref{lemma:neighborhood-of-inverse-image} to their union.
    After intersecting these neighborhoods with $\mathrm{inf}_G^{-1}(V_{\theta,\epsilon})$, we obtain (1) and (2).
    We may repeat this construction whenever we shrink the neighborhoods below.
    Suppose that $\pi_i$ corresponds to a regular unitary discrete datum $(M,\delta)$.
    We write a nearby discrete datum as $(M,\delta_u \otimes \chi_\nu)$, where $\delta_u$ is unitary and $\nu$ is real.
    By regularity, we may assume that $\delta_u$ has trivial stabilizer in the relative Weyl group.
    If the datum corresponds to an irreducible unitary representation, the Langlands classification applied to its Hermitian dual gives a Weyl element $w$ such that $w\delta_u \simeq \delta_u$ and $w\nu=-\nu$.
    Thus $w=1$ and $\nu=0$, so this representation is good.

    We choose a compact open subgroup $K'$ of $G$, depending only on $C$, such that the representations generated by their $K'$-fixed vectors form a product of Bernstein components containing all those meeting $C$ (see \cite{Renard2010-p-adiques}*{VI.10.6, Proposition (i)}).
    In particular, every standard module with infinitesimal character in $C$ is generated by its $K'$-fixed vectors.

    We use \cref{lemma:Weyl-chamber-is-deformation-retract} and the following three facts.
    \begin{enumerate}
        \item
        The cardinality of $\mathrm{inf}_{G}^{-1}(\theta)$ is bounded by $\abs{W}$.
        We order the points by the squared norm of their dominant Langlands exponent, using a $W$-invariant inner product.
        This order is preserved when we compare points over the same nearby cuspidal support $\theta'$.
        Indeed, on the unramified twists of a unitary discrete datum $(M,\delta)$, the real part of the cuspidal support of $\delta$ is orthogonal to $\fa_M^*$ because the central character of $\delta$ is unitary.
        Thus $\|\mathrm{Re}\,\theta'\|^2-\|\nu'\|^2$ is constant on this component, where $\nu'$ is the dominant Langlands exponent.
        Moreover, $\nu'_1 <_{P_0} \nu'_2$ implies $\|\nu'_1\|^2 < \|\nu'_2\|^2$, since $\nu'_1$ is dominant and $\nu'_2-\nu'_1$ is a nonzero sum of simple roots with nonnegative coefficients.
        \item
        We apply \cref{lemma:decomposition-irreducibles-to-standards} and \cref{proposition:explicit-construction-of-elements-of-hecke-algebra} to choose $h_{\theta,\epsilon} \in \Cc(G)$ with the following trace values over $\theta$.
        \begin{enumerate}
            \item At each good tempered point, we choose the trace strictly between $0$ and $\epsilon$.
            \item At each bad tempered point $\pi$, we choose the trace to be $2\dim(\pi^{K'})$.
            \item At the remaining points, we choose the traces on standard modules in increasing order of the norm of the exponent.
            At each step we choose the main term to exceed the sum of the absolute values of the lower terms by at least $2$.
            For standard modules with the same discrete datum, we use a common multiple of the dimension of their $K'$-fixed vectors, as in \cref{proposition:explicit-construction-of-elements-of-hecke-algebra}.
        \end{enumerate}
        \item 
        The regular functions in \cref{proposition:explicit-construction-of-elements-of-hecke-algebra} are continuous, so we may shrink the neighborhoods to preserve the strict bounds on the main terms.
        By replacing $h_{\theta,\epsilon}$ with $(h_{\theta,\epsilon}+h_{\theta,\epsilon}^*)/2$, we take the real part of its trace on unitary representations.
    \end{enumerate}
    The choice at the good points gives (3).
    For (4), we apply \cref{lemma:decomposition-irreducibles-to-standards} to the standard module of $\pi$.
    Each lower term has strictly smaller exponent norm already at the center of the fiber, by the norm comparison above.
    The multiplicities and the number of terms are uniformly bounded, and the dimensions of the $K'$-fixed vectors are bounded on the finitely many components meeting $\mathrm{inf}_G^{-1}(C)$.
    We may therefore choose each main term to dominate the uniform bounds from the preceding steps throughout the neighborhoods.
    This gives the lower bound $1$ and an upper bound $c_{G,C}$ depending only on $G$ and $C$.
    The choice at the bad tempered centers and continuity give a trace between $0$ and $c_{G,C}$ at good tempered points near a bad tempered center, after increasing $c_{G,C}$ if necessary.
\end{proof}

\begin{proof}[Proof of \cref{theorem:construction}]
   We follow the proof of \cite{Sau97}*{Th\'eor\`eme 5.4}, taking into account the corrections in \cite{Shi12}*{Appendix A}.
   We fix a compact neighborhood $C_1$ of $C$ and positive numbers $\epsilon_1,\epsilon_2,\epsilon_3,\epsilon_4$ to be chosen below.

   The bad tempered data form a Plancherel-null set in $\Theta_{\disc}(G)_{\temp}$.
   We choose an open neighborhood $U_{bad}$ in $\Theta_{\disc}(G)_{\temp}$ of the bad tempered data above $C$ such that
   \begin{align*}
          \mu^G(U_{bad}) < \epsilon_4.
   \end{align*}

   For each $\theta \in C$, we choose $V_{\theta,\epsilon_1} \subset C_1$ and $h_{\theta,\epsilon_1} \in \Cc(G)$ as in \cref{lemma:key-lemma}.
   On the fiber over $\theta$, we have
   \begin{align*}
      0 \leq &\widehat{h_{\theta, \epsilon_1}}(\pi) \leq \epsilon_1, \quad \pi \in \Irr_{\unit}(G)_{\mathrm{good}}, \\
      1 \leq &\widehat{h_{\theta, \epsilon_1}}(\pi) \leq c_{G, C}, \quad \pi \in \Irr_{\unit}(G)_{\mathrm{bad}},
    \end{align*}
    where $c_{G,C}$ depends only on $G$ and $C$.
    We shrink the neighborhoods of the bad tempered centers so that their tempered parts lie in $U_{bad}$.
    We also take the neighborhoods of the non-tempered centers to be disjoint from the closed tempered locus.
    Thus, over $V_{\theta,\epsilon_1}$, the trace lies between $0$ and $\epsilon_1$ at good points outside $U_{bad}$ and between $0$ and $c_{G,C}$ at good points in $U_{bad}$.
    At every bad unitary point over $V_{\theta,\epsilon_1}$, it lies between $1$ and $c_{G,C}$ by \cref{lemma:key-lemma}(4).
    
    Then, by \cref{lemma:Besikovitch},  there exists a subset $\{ \theta_i \}_{i = 1, 2, ..., N(\epsilon_1)}$ of $C$ and an open covering $\{W_i\}_{i = 1, ..., N(\epsilon_1)}$ of $C$ such that  
    \begin{align*}
        \sum_{i = 1, ..., N(\epsilon_1)} \mathbf{1}_{\overline{W_i}} 
        &\leq   
        c'_G \mathbf{1}_{C_1}, \\  
        \overline{W_i}
        &\subset  
        V_{\theta_i, \epsilon_1},
    \end{align*}
    where $\mathbf{1}_{Z}$ denotes the characteristic function of $Z$ and $c'_G$ is a positive constant which only depends on $G$.

    We set 
    \begin{align*}
        M(\epsilon_1) = \sup_{i=1, ..., N(\epsilon_1), \pi \in \Irr_{\unit}(G)} \abs{\widehat{h_{\theta_i, \epsilon_1}}(\pi)}.
    \end{align*}

    By \cite{Sau97}*{Lemme 3.4}, we may choose a vertical band $B$ with compact real part containing $C_1$ such that $\widehat{h_{\theta_i,\epsilon_1}}(\pi)=0$ for every $i$ whenever $\pi \in \Irr_{\unit}(G)$ and $\inf_G(\pi) \notin B$.
    We choose an open neighborhood $W'_i$ of $\overline{W_i}$ in $V_{\theta_i,\epsilon_1}$ for each $i$ such that
    \begin{align*}
         {\mathrm{inf}_G}_*(\mu^G)(\bigcup_{i=1,..., N(\epsilon_1)} (W_i'\setminus \overline{W_i})\, ) \leq \epsilon_2.
    \end{align*}
    This follows from the outer regularity of  the measure ${\mathrm{inf}_G}_*(\mu^G)$, see \cite{Rud87}*{Theorem 2.18}.
    
    We can choose a function $a_i \in \cA(G)$ as follows by applying \cite{Sau97}*{Lemme 5.2} or the Paley--Wiener theorem as in \cite{ClozelDelorme1990} and \cite{BernsteinDeligneKazhdan1986}: The function $a_i$ is real on $\Theta(G)_{\herm}$ and we have 
    \begin{align*}
        0 \leq a_i(\theta) &\leq 1+\epsilon_3, \quad \theta \in B \cap \Theta(G)_{\herm}, \\ 
        a_i(\theta) &\geq 1, \quad \theta \in \overline{W_i} \cap \Theta(G)_{\herm},  \\
        a_i(\theta) &\leq \epsilon_3, \quad \theta \in (B \setminus W'_i) \cap \Theta(G)_{\herm}.
    \end{align*}
    
    We consider the function 
    \begin{align*} 
       \sum_{i = 1, ..., N(\epsilon_1)} (\mathrm{inf}_G^*a_i) \widehat{h_{\theta_i, \epsilon_1}}.
    \end{align*}
    By the module property, this function is of the form $\widehat{h_1} \in \cF(G)$.
    We bound the absolute values of the summands according to the location of the infinitesimal character. 
    \begin{itemize}
    \item 
    At a good point, the sum over $i$ with $\inf_G(\pi) \in \overline{W_i}$ is bounded by 
    \begin{align*}
        (1+\epsilon_3)c'_G \bigl(\epsilon_1 (\mathbf{1}_{C_1} \circ \mathrm{inf}_G) + c_{G,C} \mathbf{1}_{U_{bad}}\bigr).
    \end{align*}
    \item   
    At a bad point, the sum over the same indices is bounded by 
    \begin{align*}
        (1+\epsilon_3) c_{G,C} c'_G (\mathbf{1}_{C_1} \circ \mathrm{inf}_G) (\mathbf{1}_{\Irr_{\unit}(G)_{\mathrm{bad}}})
    \end{align*}   
    \item    
    The sum over $i$ with $\inf_G(\pi) \in W'_i \setminus \overline{W_i}$ is bounded in absolute value by 
    \begin{align*}
        (1+\epsilon_3)M(\epsilon_1) N(\epsilon_1) (\mathbf{1}_{\bigcup_{i=1,..., N(\epsilon_1)} (W_i'\setminus \overline{W_i})} \circ \mathrm{inf}_G)
    \end{align*}   
    \item      
    The remaining summands contribute at most 
    \begin{align*}
       \epsilon_3 \sum_{i}\abs{\widehat{h_{\theta_i, \epsilon_1}}}.
    \end{align*}
  \end{itemize}

  We take a positive function $\widehat{h_2}$ with $\widehat{h_2} \geq 1$ on $C_1$ by applying \cite{Sau97}*{Lemme 5.3}.
  We set 
  \begin{align*}
     \widehat{h'} = \widehat{h_1} + \epsilon_3 M(\epsilon_1) N(\epsilon_1) \widehat{h_2}.
  \end{align*}
  By considering whether $\pi \in \inf_G^{-1}(V_{\theta_i, \epsilon_1})$ or not, it is easy to see that we have $\widehat{h'} \geq 0$ on $C_1$.
  Similarly, on the set $\inf_G^{-1}(C) \cap \Irr_{\unit}(G)_{\mathrm{bad}}$, we have $\widehat{h'} \geq 1$.
  We first choose $\epsilon_1$ and $\epsilon_4$ small using the bounded overlap and the measure of $U_{bad}$.
  Once the covering and the local functions are fixed, we choose $\epsilon_2$ and $\epsilon_3$ small enough that the preceding bounds give 
  \begin{align*}
     \mu^G(\abs{\widehat{h'}}) < \epsilon/2.
  \end{align*} 
  We now follow the argument starting from \cite{Sau97}*{p.178, line 12}.
  We can find a function $h'' \in \Cc(G)$ with 
  \begin{align*}
     \widehat{h''} \geq \abs{\widehat{h'}}
  \end{align*}
  on $\Irr_{\unit}(G)$.
  We take a positive constant $\epsilon_5$ such that we have 
  \begin{align*}
     \epsilon_5\mu^G(\widehat{h''}) \leq \frac{\epsilon}{3}.
  \end{align*}
  By \cite{Sau97}*{Lemme 3.4}, we may enlarge $B$ so that $\widehat{h'}$ vanishes on unitary representations with infinitesimal character outside $B$.
  We apply the Stone--Weierstrass theorem on the Hermitian part of this band to choose $b \in \cA(G)$ with the following properties. 
  \begin{enumerate}
    \item 
    $1 \leq b \leq \frac{4}{3}$ on $C \cap \Theta(G)_{\herm}$,
    \item
    $0 \leq b \leq \frac{4}{3}$ on $C_1 \cap \Theta(G)_{\herm}$, and 
    \item
    $\abs{b} \leq \epsilon_5 \leq \frac{4}{3}$ on $(B \cap \Theta(G)_{\herm}) \setminus C_1$.
  \end{enumerate}
  We set 
  \begin{align*}
     h = b h' + \epsilon_5 h''.
  \end{align*}
  Then, on $\mathrm{inf}_G^{-1}(C_1) \cap \Irr_{\unit}(G)$, we have 
  \begin{align*}
     \widehat{h} \geq  b \widehat{h'} \geq 0.
  \end{align*}
  Also, on $\Irr_{\unit}(G) \setminus \mathrm{inf}_G^{-1}(C_1)$, we have 
  \begin{align*}
      \widehat{h} \geq \epsilon_5 \widehat{h''} - \abs{b \widehat{h'}} \geq 0.
  \end{align*}
  Thus, we have 
  \begin{align*}
      \widehat{h} \geq 0
  \end{align*}
  on $\Irr_{\unit}(G)$.
  We also have 
  \begin{align*}
     \mu^G(\widehat{h}) \leq \frac{4}{3} \cdot \frac{\epsilon}{2} + \frac{\epsilon}{3}  = \epsilon.
  \end{align*}
  Hence the result.
\end{proof}

\begin{remark}
     The similar method can be applied to proving the proposition in case $F = \R$, by replacing the role of $K'$ with minimal $K$-type of $i^G_{P}(\sigma)$ for each discrete datum $\{M, \sigma\}_G$, see \cite{Sau97}*{Lemme 5.1}.
\end{remark}

\section{Applications}

\subsection{Some lemmas}

Let $\bG$ be a reductive group scheme over a product $A$ of local fields of characteristic $0$. 
We set $G = \bG(A)$.
Let $\mu^G$ denote the Plancherel measure of $G$.
Let $\{ \mu^G_n \}_{n \geq 1}$ be a sequence of positive Borel measures on $\Irr_{\unit}(G)$.

\begin{lemma}[\cite{Sau97}*{Proposition 1.3}]\label{lemma:general}
  Let $\mu^G_n$ be a sequence of positive Borel measures on $\Irr_{\unit}(G)$.
  We assume that for any $h \in \Cc(G)$, we have 
  \begin{align*}
      \lim_{n \to \infty} \mu^G_n(\widehat{h}) = \mu^G(\widehat{h}).
  \end{align*} 
  Then, it follows that
  \begin{align*}
    \lim_{n \to \infty} \mu^G_n(\phi) = \mu^G(\phi)
  \end{align*}
  for any $\phi \in \mathcal{RI}(\Irr_{\temp}(G), \mu^G)$.
\end{lemma}

\begin{proof}
    Let $\epsilon$ be any positive real number.
    By \cref{theorem:density-principle}, we can find functions $h_1, h_2 \in \Cc(G)$ such that
    \begin{align*}
        \abs{\phi - \widehat{h_1}} \leq \widehat{h_2}, \\
        \mu^G(\widehat{h_2}) < \epsilon.
    \end{align*}
    We have 
    \begin{multline*}
        \abs{\mu^G_n(\phi) - \mu^G(\phi)} 
        \leq
        \abs{\mu^G_n(\phi) - \mu^G_n(\widehat{h_1})} 
        + 
        \abs{\mu^G(\phi) - \mu^G(\widehat{h_1})} \\
        +
        \abs{\mu^G_n(\widehat{h_1}) - \mu^G(\widehat{h_1})}.
    \end{multline*}
    By the assumption on $h_1, h_2$, we have 
    \begin{align*}
        \abs{\mu^G_n(\phi) - \mu^G_n(\widehat{h_1})} 
        \leq    
        \mu^G_n(h_2), \\
        \abs{\mu^G(\phi) - \mu^G(\widehat{h_1})} 
        \leq    
        \mu^G(h_2).
    \end{align*}
    By substituting these inequalities into the previous inequality and considering the limit superior, we obtain 
    \begin{align*}
        \limsup_{n \to \infty} 
        \abs{\mu^G_n(\phi) - \mu^G(\phi)} 
        \leq    
        2\mu^G(\widehat{h_2})
        < 2\epsilon.
    \end{align*} 
    As we can take $\epsilon > 0$ arbitrarily, we obtain 
    \begin{equation*}
        \lim_{n \to \infty} \mu^G_n(\phi) = \mu^G(\phi). \qedhere
    \end{equation*}
\end{proof}

We also give a lemma to globalize local characters.
Let $\Lambda = \Lambda_1 \coprod \Lambda_2$ be a set such that $\Lambda_2$ is finite.
For $\lambda \in \Lambda_1$, let $\cP_{\lambda} = (G_{\lambda}, K_{\lambda})$ be a pair consisting of a locally compact Hausdorff group $G_{\lambda}$ and its open compact subgroup $K_{\lambda}$.
For $\lambda \in \Lambda_2$, let $\cP_{\lambda} = G_{\lambda}$ be a locally compact Hausdorff space.
Let $G$ be a restricted product $
\prod_{\lambda \in \Lambda_1}' (G_{\lambda}, K_{\lambda}) \times \prod_{\lambda \in \Lambda_2} G_{\lambda}$ of the family $\cP_{\lambda}$.
Then, $G$ is locally compact.
Let $\Gamma$ be a discrete subgroup of $G$ such that the map $\Gamma \hookrightarrow G \twoheadrightarrow G_{\lambda_0}$ is injective for an element $\lambda_0 \in \Lambda$.

\begin{lemma}\label{lemma:globalization-compact}
    Let $S$ be a finite subset of $\Lambda_1$ which does not contain $\lambda_0$. 
    Let $\chi_{S}$ be a continuous character of $K_S = \prod_{\lambda \in S} K_{\lambda}$.
    Then, there exists a character $\chi$ of $G/\Gamma$, such that we have
    \begin{align*}
        \chi_{\restriction_{K_{S}}}&= \chi_S, \\
        \chi_{\restriction{{\prod_{\Lambda_1 \setminus (S \cup \{ \lambda_0 \})} K_{\lambda}}}}
        &= 
        \mathbf{1}_{\restriction{{\prod_{\Lambda_1 \setminus (S \cup \{ \lambda_0 \})} K_{\lambda}}}}.
    \end{align*}
\end{lemma}

\begin{proof}
    By the hypothesis on $\lambda_0$, we have 
    \begin{align*}
        \Gamma \cap \prod_{\lambda \neq \lambda_0} K_{\lambda} = \{1\}.
    \end{align*}
    Also, the group $\prod_{\lambda \neq \lambda_0} K_{\lambda}$ is compact, the image of the map
    \begin{align*}
        \prod_{\lambda \neq \lambda_0} K_{\lambda} \to G/\Gamma.
    \end{align*}
    is compact and closed.
    Thus, the subgroup $\Gamma \prod_{\lambda \neq \lambda_0} K_{\lambda}/\Gamma$ of $G/\Gamma$ is closed and isomorphic to $\prod_{\lambda \neq \lambda_0} K_{\lambda}$.
    By the Pontryagin duality, there exists a unitary extension $\chi$ of the character $\chi_S \otimes \mathbf{1}_{\restriction{{\prod_{\Lambda_1 \setminus (S \cup \{ \lambda_0 \})} K_{\lambda}}}}$ to $G/\Gamma$. 
\end{proof}

\begin{corollary}\label{corollary:globalization-characters}
    Let $T$ be a torus over a number field $k$, and let $S$ be a finite set of finite places of $k$. 
    Suppose that $\chi_S$ be a character of the maximal compact subgroup $T(k_S)_c$ of $T(k_S)$.
    Then, there exists a unitary character $\chi'$ of $T(\A_k)/T(k)$ which extends $\chi_S$ on $T(k_S)_c$ and unramified outside $S$. 

    We assume that $S$ is a singleton $\{ u_0 \}$ and $A_{T_{u_0}} = A_T \times_{k} k_{u_0}$.
    Let $\chi_{u_0}$ be a unitary character of $T(k_{u_0})$.
    Then, there exists a unitary character $\chi$ of $T(\A_k)/T(k)$ which extends $\chi_{u_0}$ on $T(k_{u_0})$ and unramified outside $\{u_0\}$. 
\end{corollary}

\begin{proof}
     The first statement follows from the previous lemma by taking the set of places of $k$ as $\Lambda$, the set of finite places of $k$ as $\Lambda_1$,  the set of infinite places of $k$ as $\Lambda_2$ and an infinite place $k$ as $\lambda_0$.
     We now consider the second statement.
     By applying the first statement to obtain a global character $\chi'$ which extends ${\chi_{u_0}}_{\restriction{T(k_{u_0})_c}}$.
     Considering the twist by $\chi'$, we may assume that ${\chi_{u_0}}_{\restriction{T(k_{u_0})_c}} = 1$.
     Note that the group $T(k_{u_0})_c$ is the intersection of the kernels of unramified characters of $T(k_{u_0})$.
     Thus, the character $\chi_{u_0}$ is unramified.
     By the assumption $A_{T_{u_0}} = A_T \times_{k} k_{u_0}$, we have the injection 
     \begin{align*}
        T(k_{u_0})/T(k_{u_0})_c \hookrightarrow 
        T(\A_k)/T(\A_k)^1,
     \end{align*}
     and the unramified character $\chi_{u_0}$ can be extended to an automorphic character $\chi$ of $T(\A_k)$ by the Pontryagin duality. 
     Hence the result.
\end{proof}

\subsection{Shin's automorphic density theorem}
Let $k$ be a number field with a real place $u_0$. 
Thus, we have $k_{u_0} \cong \R$.
Let $G$ be a connected reductive group over $k$.
We assume that the center $Z_G$ of $G$ is an induced torus over $k$.
For a finite set $S$ of places of $k$, we denote the group $\prod_{v \in S}G(k_v)$ by $G_S$.
Similarly, we set $G^S = \prod_{v \not\in S}' G(F_v)$.
 
Let $\lambda$ be an automorphic unitary character of $Z_G(\A_k)$.
As we have $Z_G(\A_k) = (Z_G)_S \times Z_G^S$, we can decompose this central character as $\lambda = \lambda_S \boxtimes \lambda^S$.
By \cite{MoeglinWaldspurger2016-Stablisation2}*{VI 2.8}, we have the invariant trace formula with the fixed central character $\lambda$.

\begin{definition}
    Let $\Irr_{\unit}(G(F_S), \lambda_S)$ be a set of unitary representations of $G(F_S)$ with the central character $\lambda_S$.
    This set is also endowed with the restriction of the Fell topology on $\Irr_{\unit}(G(F_S))$.
    A Borel subset $U_S$ of $\Irr_{\unit}(G_S, \lambda_S)$ is said to be $\mu^{G_S}$-regular if the measure $\mu^{G_S}(\partial U_S)$ of the boundary is equal to zero.
\end{definition}

\begin{theorem}[c.f. \cite{Shi12}*{Theorem 4.11}]\label{theorem:globalizaton-Shin}
   Let $k$ be a number field with a real place $u_0$.
   Let $G$ be a connected reductive group over $k$.
   Let $\lambda$ be an automorphic unitary character of $Z_G(\A_k)$.  
   Let $S$ be a finite set of places of $k$ containing all the infinite places and assume that $G$ and $\lambda$ is unramified outside $S$.
   We set $S' = S \setminus \{u_0\}$.

   We assume the following. 
   \begin{itemize}
    \item    
     The center $Z_G$ of $G$ is induced.
    \item
     The group $G(k_{u_0})$ has a discrete series 
     representation with character $\lambda_{u_0}$ and also we have $A_G \times_{k} k_{u_0} = A_{G_{k_{u_0}}}$.
     \item        
     We are given a $\mu^{G_{S'}}$-regular relatively quasi-compact set $U_{S'} \subset \Irr_{\unit}(G(k_{S'}), \lambda_{S'})$ such that $\mu^{G_{S'}}(U_{S'}) > 0$.
   \end{itemize}
   Then, there exists a cuspidal automorphic representation $\Pi$ of $G(\A_k)$ with the central character $\lambda$ such that 
   \begin{itemize}
     \item  we have $\Pi_{S'} \in U_{S'}$, 
     \item the representation $\Pi_v$ is unramified for $v \not\in S$, and 
     \item the representation $\Pi_{u_0}$ is a discrete series representation of $G(k_{u_0})$ and the infinitesimal character of $\Pi_{u_0}$ is sufficiently regular.
   \end{itemize}
\end{theorem}

\begin{proof}
   The proof is essentially the same as \cite{Shi12}*{Theorem 4.11}.
   We use the invariant trace formula for the fixed central character $\lambda$ constructed in \cite{MoeglinWaldspurger2016-Stablisation2}*{VI 2.8}, instead of the usual Arthur trace formula.

   We first recall some definitions which are similar to those in \cite{Shi12}.
   We set $G^{\natural} = G/Z_G$.
   Let $\Pi_{\cusp}(G(\A_k), \lambda)$ denote the set of cuspidal automorphic representations of $G(\A_k)$ with the central character $\lambda$ and $m_{\cusp}(\pi)$ be the multiplicity of $\pi$ in the space of cuspidal automorphic forms on $G(\A_k)$.
   Let $\{ \xi_n \}_{n \geq 1}$ be a sequence of the finite dimensional algebraic representation of $G(k_{u_0} \times_{\R} \C)$ with the central character $\lambda_{u_0}$ on $Z_G(k_{u_0})$, such that $\lim_{n \to \infty} \xi_n = \infty$ in the sense of \cite{Shi12}*{Definition 3.5}.
   We define the functional $\mu^{\cusp}_{\phi^{S'}, \xi_n, \lambda}$ on $\Cc(G(F_{S'}), \lambda_{S'}^{-1})$ as 
   \begin{multline*}
    \mu^{\cusp}_{\phi^{S'}, \xi_n, \lambda}(f)  \\
    = 
    \frac{(-1)^{q(G_{u_0})}}{\dim(\xi_n)\vol(G^{\natural}(k) \backslash G^{\natural}(\A_k))}
    \sum_{ \pi \in \Pi_{\cusp}(G(\A_k), \lambda) } 
    m_{\cusp}(\pi) \tr\pi(f \otimes \phi^{S} \otimes \phi_{\xi_n}).
   \end{multline*}
   Here, the function $\phi_{\xi_n} \in \Cc(G(k_{u_0}), \lambda_{u_0}^{-1})$ is the Euler-Poincare function associated with $\xi_n$ and $\phi^{S}$ is the unit of the ring of unramified functions in $\Cc(G(\A^S_k), (\lambda^S)^{-1})$.
   Let $\mathcal{FT}(G(F_{S'}), \lambda_{S'}^{-1})$ denote the space of Fourier transforms of functions in $\Cc(G(k_{S'}), \lambda_{S'}^{-1})$.
   The last expression is equal to 
   \begin{multline*}
       \frac{\vol(K^SZ_G(\A_k)/Z_G(\A_k))}{\dim(\xi_n)\vol(G^{\natural}(k) \backslash G^{\natural}(\A_k))} \\
       \times        
       \sum_{ \pi \in \Pi_{\cusp}(G(\A_k), \lambda), \pi^{K^S} \neq 0, \text{$\xi_n$-cohomological}} 
    m_{\cusp}(\pi) \tr\pi_{S'}(f) 
   \end{multline*}
   if $n$ is large.

   The functional $\mu^{\cusp}_{\phi^{S'}, \xi_n, \lambda}$ factors through the map $\Cc(G(k_{S'}), \lambda_{S'}^{-1}) \to \mathcal{FT}(G(F_{S'}), \lambda_{S'}^{-1})$ taking the Fourier transform. 
   We denote the resulting functional on $\mathcal{FT}(G(F_{S'}), \lambda_{S'}^{-1})$ by $\widehat{\mu}^{\cusp}_{\phi_n^{S'}, \xi, \lambda}$.
   Applying \cite{Shi12}*{Lemma 4.9} and $\phi_{\xi_n}(1) = \dim(\xi_n)$, we have 
   \begin{align*}
    \lim_{n \to \infty} 
    \widehat{\mu}^{\cusp}_{\phi^{S'}, \xi_n, \lambda}(\widehat{f}) 
    =  
    f(1) = \mu^{S'}(\widehat{f}),
   \end{align*}     
   for any $f \in \Cc(G(F_{S'}), \lambda_{S'}^{-1})$.
   The assumption $A_{G} \times_{k} k_{u_0} = A_{G_{u_0}}$ is used here to apply the result \cite{GKM1997}*{Theorem 5.2}.
   An analog of the Sauvageot density principle for the fixed central character $\lambda^{-1}_{S'}$ proved in \cite{Binder2019}*{Proposition 3.3.2} by using \cref{theorem:density-principle}.
   Hence, the computation above and an analog of \cref{lemma:general} implies the result.
\end{proof}

\begin{remark}
    Instead of assuming that the center $Z_G$ is induced, we can prove a  similar result by assuming that $G_v$ has an anisotropic center over $k_v$ for $v \in S$, by considering the Poisson summation formula and finiteness of the class number for $Z_G$.
\end{remark}

\subsection{Globalization to spherical cusp forms}

Let $k$ be a number field.
We set $G = \PGL_n(\A_k)$.
Let $S_{\infty}$ be the set of infinite places of $k$.
Let $A_0$ be a maximal split torus of $G$ over $k$.
We denote the group $W(G, A_0)$ by $W_{\R}$. 
Note that the space of spherical tempered representations $\Irr_{\temp, \mathrm{sph}}(G_{S_{\infty}})$ is an open and closed subset of $\Irr_{\temp}(G_{S_{\infty}})$.
This space is identified with $i\fa^{*}_{0}/W_{\R}$. Thus, for any subset $\Omega$ and positive number $t \in \R_{>0}$, we have the dilatation $t\Omega$ of $\Omega$.

\begin{theorem}
   For any finite set of finite places $S$,  $\mu^{G_S}$-regular relatively quasi-compact set $U_S$ of $\Irr_{\unit}(G_S)$ with positive measure and any open subset $U_{\infty}$ of $\Irr_{\temp, \mathrm{sph}}(G_{S_{\infty}})$ with positive measure, there exists a cuspidal automorphic representation $\Pi = \bigotimes_{v}^{'} \Pi_v$ of $\PGL_n(\A_k)$ such that 
   \begin{itemize}
    \item we have $\Pi_S \in U_S$,
    \item $\Pi_v$ is unramified for $v \not \in S \cup S_{\infty}$, and
    \item $\Pi_{S_{\infty}}$ is spherical and the representation is in the set $tU_{\infty}$ for some $t \geq 1$.
   \end{itemize}
\end{theorem}

\begin{proof}
    This follows from \cite{Eikemeier2022}*{Theorem 1.2} and \cref{theorem:density-principle} and we will give the details below.
    The first result is an extension of the result  \cite{FinisMatz2021}*{Theorem 1.1} to quasi-split groups, see the footnote of \cite{FinisMatz2021}*{p.1038}.

    Now, we give the details. 
    Let $T_0$ be maximal split torus of $G$ and $\ft_0$ be its Lie algebra.
    Then, the infinitesimal characters of spherical representations of $G_{S_{\infty}}$ are parametrized by the quotient $(\ft^{*}_0 \otimes_{\R} \C)/W_{\R}$. 
    Those for unitary spherical representations are in the subset $i\ft^{*}_0/W_{\R}$.
    For any bounded set $B \subset (\ft^{*}_0 \otimes_{\R} \C)/W_{\R}$ and for any smooth function $\tau \in \Cc(G(\A_{k, \mathrm{fin}}))$, we set 
    \begin{align*}
          m(B, \tau) = \sum_{\lambda \in B} \sum_{\pi \in \Pi_{\disc}(G(\A_k)), \pi_{\infty}^{K_{\infty}} \neq 0, W_{\R}\lambda_{\pi_{\infty}} = W_{\R}\lambda} m_{\disc}(\pi) \tr \pi_{\mathrm{fin}}(\tau).
    \end{align*}
    This counts the contribution of spherical discrete automorphic representations with the infinitesimal character in the $W_{\R}$-orbit of $B$.
    Let $\Omega \subset i\ft^{*}_0$ be a bounded open subset with the rectifiable boundary in the sense of  \cite{FinisMatz2021}*{p.1094, line 15}.
    By shrinking $\Omega$, we can always achieve this condition.
    For $t \geq 1$, let $\Lambda_{\Omega}(t)$ be the volume of $t\Omega$ with respect to the Plancherel measure, see \cite{Eikemeier2022}*{p.14, line 25} for an explicit description.
    Then, by \cite{Eikemeier2022}*{Theorem 1.2}, we have 
    \begin{align*}
        \lim_{t \to \infty} \frac{1}{\Lambda_{\Omega}(t)}  m(t \Omega, \tau)  
        = 
        \tau(1)
        =
        \mu^{G_S}(\widehat{\tau_S}).
    \end{align*}

    We take $\tau = \tau_S \otimes \bigotimes_{v \not \in S}\mathbf{1}_{K_v}$.
    Note that if we write $m(B, \tau_S)$ for $m(B, \tau)$, then, we have
    \begin{align*}
        m(B, \tau_S)
        &=
        \sum_{\lambda \in B} 
        \sum_{\pi \in \Pi_{\disc}(G(\A_k)), \pi^{K^SK_{\infty}} \neq 0, W_{\R}\lambda_{\pi_{\infty}} = W_{\R}\lambda} 
        m_{\disc}(\pi) \tr \pi_{S}(\tau_S) \\
        &= 
        \sum_{\lambda \in B} \sum_{\pi \in \Pi_{\disc}(G(\A_k)), \pi^{K^SK_{\infty}} \neq 0, W_{\R}\lambda_{\pi_{\infty}} = W_{\R}\lambda} m_{\disc}(\pi) \widehat{\tau_S}(\pi_S).
  \end{align*}
   Thus, if for any $F \in \mathcal{RI}(\Irr_{\unit}(G_S), \mu^{G_S})$ we set 
   \begin{align*}
     \mu^{G_S}_{\disc, t}(F) =
     \frac{1}{\Lambda_{\Omega}(t)} 
     \sum_{\lambda \in B} \sum_{\pi \in \Pi_{\disc}(G(\A_k)), \pi^{K^SK_{\infty}} \neq 0, W_{\R}\lambda_{\pi_{\infty}} = W_{\R}\lambda} m_{\disc}(\pi) F(\pi_S),
   \end{align*}
   then, we have 
        \begin{align*}
        \lim_{t \to \infty} \mu^{G_S}_{\disc, t}(\widehat{\tau_S}) = \mu^{G_S}(\widehat{\tau_S}).
    \end{align*}
   Then, \cref{lemma:general} implies the result.   
   Note that we can replace $m_{\disc}(\pi)$ with $m_{\cusp}(\pi)$, see \cite{FLM15}*{Remark 1.1}.
\end{proof}

\begin{remark}
    The argument of this type can be applied to more general (not necessarily adjoint) quasi-split reductive groups, see \cite{Eikemeier2022}*{Theorem 1.2}.
\end{remark}

\section*{Declarations}

\subsection*{Conflict of interest}
The authors declare that they have no competing interests.

\subsection*{Data Availability}
No datasets were generated or analysed during the current study.

\bibliographystyle{plain}
\bibliography{sauvageot}

\end{document}